\documentclass[12pt]{article}
\usepackage{./macros}

\title{Cycle type factorizations in \(\GL_n \F_q\)}
\author{Graham Gordon}

\begin{document}
\maketitle

\abstract{
Recent work by Huang, Lewis, Morales, Reiner, and Stanton suggests that the regular elliptic elements of \(\GL_n \F_q\) are somehow analogous to the \(n\)-cycles of the symmetric group.
In 1981, Stanley enumerated the factorizations of permutations into products of \(n\)-cycles.
We study the analogous problem in \(\GL_n \F_q\) of enumerating factorizations into products of regular elliptic elements.
More precisely, we define a notion of cycle type for \(\GL_n \F_q\) and seek to enumerate the tuples of a fixed number of regular elliptic elements whose product has a given cycle type.
In some special cases, we provide explicit formulas, using a standard character-theoretic technique due to Frobenius by introducing simplified formulas for the necessary character values.
We also address, for large \(q\), the problem of computing the probability that the product of a random tuple of regular elliptic elements has a given cycle type.
We conclude with some results about the polynomiality of our enumerative formulas and some open problems.
}

\section{Introduction}

Factorization enumeration has a long history filled with interesting combinatorics and topology \cite{denes, elsv, cacti, jackson_top, vakil}.
For example, in \cite{ncycle}, Stanley enumerates the ordered factorizations of an arbitrary permutation in \(\FS_n\) into a product of \(n\)-cycles.
We are interested in finding an analogue of Stanley's result for the finite general linear group \(\GL_n \F_q\).

We assume some basic knowledge of the representation theory of \(\FS_n\).
For each partition \(\mu \vdash n\),
let \(\CC_\mu \subset \FS_n\) denote the conjugacy class consisting of permutations with cycle type \(\mu\).
Let \(m_i (\mu) \) denote the multiplicity of \(i\) in \(\mu\).
For \(\lambda \vdash n\), let \(\chi^\lambda_\mu\) denote the irreducible character \(\chi^\lambda\) of \(\FS_n\) corresponding to \(\lambda\) evaluated on an element of \(\CC_\mu\).
Let \(\N = \{1,2,3,\ldots\}\) denote the positive integers.
For any \(\mu \vdash n\) and \(k \in \N\), define
\begin{equation}
g_{k,\mu} = \# \{(t_1,\ldots,t_k) \in \CC_{(n)}^k : t_1\cdots t_k \in \CC_\mu\}.
\end{equation}
The quantity \(g_{k,\mu}\) is \(\#\CC_\mu\) times as large as the aforementioned quantity Stanley computes.
\begin{theorem}[Stanley {\cite[Thm.~3.1]{ncycle}}]
\label{inspiration}
For all \(n,k \in \N\) and \(\mu \vdash n\), the number of ordered \(k\)-tuples of \(n\)-cycles whose product equals an arbitrary fixed permutation of cycle type \(\mu\) is
\begin{equation}
\label{chi_phrasing_of_insp}
\frac{g_{k,\mu}}{\# \CC_\mu}
=
\frac{(n-1)!^{k-1}}{n}
\sum_{r=0}^{n-1} \frac{(-1)^{rk} \chi^{(n-r,1^r)}_\mu}{\binom{n-1}{r}^{k-1}},
\end{equation}
and, more explicitly,
\begin{equation}
\label{explicit_phrasing_of_insp}
\frac{g_{k,\mu}}{\# \CC_\mu}
=
\frac{(n-1)!^{k-1}}{n}
\sum_{r=0}^{n-1}
\frac{(-1)^{rk}}{\binom{n-1}{r}^{k-1}}
\sum_{\nu \vdash r}
(-1)^{\sum_{a \geq 1} m_{2a}(\nu)}
\binom{m_1(\mu)-1}{m_1(\nu)}
\prod_{j = 2}^{r}
\binom{m_j(\mu)}{m_j(\nu)}.
\end{equation}
\end{theorem}

Theorem~\ref{inspiration} is proved using a character-theoretic technique due to Frobenius which we describe in Section~\ref{approach}.
The simplicity of \eqref{chi_phrasing_of_insp} comes from the fact that \(\chi^\lambda_{(n)} = 0\) unless \(\lambda\) is a \myemph{hook}, i.e., \(\lambda = (n-r,1^r)\) for some \(r \in \{0,\ldots,n-1\}\).
The more explicit phrasing \eqref{explicit_phrasing_of_insp} is obtained by evaluating the hook character values explicitly \cite[Lem.~2.2]{ncycle}.
For the sake of brevity, we will express formulas using the symbols \(\chi^{(n-r,1^r)}_\mu\) with the understanding that there exists an explicit formula for computing such character values.

For further history, see the work of Bertram-Wei \cite{bertram_wei}, Boccara \cite{boccara}, and Walkup \cite{walkup}.
Walkup developed a recursion for the number of ways to factor a given permutation in \(\FS_n\) into the product of two \(n\)-cycles.
Boccara expressed the number of factorizations of a given permutation in \(\FS_n\) into the product of an \(n\)-cycle and an \(m\)-cycle in terms of the definite integral of a certain polynomial, for \(2 \leq m \leq n\).
Bertram-Wei developed some recursive and some explicit formulas for the number of ways to factor a given permutation in \(\FS_n\) into the product of an \(n\)-cycle and an \(m\)-cycle for various \(m \leq n\).
Stanley's result of course only applies to factoring permutations into products of \(n\)-cycles, but it applies to cases with more than two factors, unlike the other results mentioned.

Let \(\GL_n \F_q\) denote the group of \(n \times n\) invertible matrices with entries in the finite field \(\F_q\) with \(q\) elements.
Consider the \myemph{regular elliptic} elements of \(\GL_n \F_q\), which are those matrices whose characteristic polynomial is irreducible over \(\F_q\).
Recent work by Huang, Lewis, Morales, Reiner, and Stanton \cite{HLR, LM, refl_fact} suggests that the regular elliptic elements are analogous to the \(n\)-cycles in \(\FS_n\) from the perspective of enumerating factorizations.

Given a matrix \(g \in \GL_n \F_q\), we define the \myemph{cycle type} of \(g\) to be \(\mu = (\mu_1,\ldots,\mu_\ell) \vdash n\) if the degrees of the irreducible factors of the characteristic polynomial of \(g\) are \(\mu_1,\ldots,\mu_\ell\) in weakly decreasing order,
and we write \(\type (g) = \mu\).
For each \(\mu \vdash n\), let \(\CT_\mu (q) \subset \GL_n \F_q\) denote the subset of matrices of cycle type \(\mu\).
In particular, \(\CT_{(n)} (q)\) is the set of regular elliptic elements.
Note that \(\{\CT_\mu (q) : \mu \vdash n\}\) is a partition of \(\GL_n \F_q\).
This definition is built on the statement made by Stong in the conclusion of \cite{stong} that a degree-\(m\) divisor of the characteristic polynomial of a matrix in \(\GL_n \F_q\) is the analog of a cycle of length \(m\) in a permutation in \(\FS_n\).
One can obtain the normalized generating function for our definition of cycle type from Stong's \cite{stong} generalization
\(\tilde{Z}_n (q; x)\)
of Kung's \cite{kung} cycle index
\(Z(\GL_n \F_q; x)\)
by making the variable substitution
\(x_{p,b} \mapsto x_{m}^{j}\)
if \(p\) has degree \(m\) and \(b \vdash j\).
The substitution \(x_{p,b} \mapsto x_m^j\) is mentioned explicitly by Fulman in Section 5 of \cite{fulman}.

Toward the end of enumerating factorizations into products of regular elliptic elements, for any \(\mu \vdash n, k \in \N\), and prime power \(q\), we define
\begin{equation}
g_{k,\mu} (q)
=
\# \{ (t_1,\ldots,t_k) \in \CT_{(n)} (q)^k : t_1 \cdots t_k \in \CT_\mu (q) \}.
\end{equation}
In this paper, we consider the quantity \(g_{k,\mu}(q)\) to be a \(\GL_n \F_q\)-analogue of \(g_{k,\mu}\),
and we seek efficient ways to compute \(g_{k,\mu}(q)\).
We also consider the \myemph{regular semisimple} elements of \(\GL_n \F_q\), which are those matrices whose characteristic polynomial has no repeated irreducible factors.
Let \(\CT_\mu^\Box (q)\) denote the set of regular semisimple elements with cycle type \(\mu\).
We explain this choice of notation in Section~\ref{cycle_rs_re}.
Note that \(\{\CT_\mu^\Box (q) : \mu \vdash n\}\) is not a partition of \(\GL_n \F_q\) in general, as not all elements of \(\GL_n \F_q\) are regular semisimple.
However, as the following result implies, for large \(q\), an arbitrarily large proportion of \(\GL_n \F_q\) is regular semisimple.
Let \(z_\mu\) denote the cardinality of the centralizer of an element of \(\CC_\mu\).

\begin{corollary}[to~Cor.~\ref{rss_cc_and_cycle_type_sizes}]
\label{cycle_type_ratio_limit}
For all \(n \in \N\) and \(\mu \vdash n\),
\label{dense}
\[
\lim_{q \to \infty} \frac{\# \CT_\mu^\Box (q)}{\# \GL_n \F_q} = \lim_{q \to \infty} \frac{\# \CT_\mu (q)}{\# \GL_n \F_q} = \frac{1}{z_\mu}.
\]
\end{corollary}

Therefore, we are also interested in enumerating factorizations of regular semisimple elements.
For any \(k \in \N, \mu \vdash n\), and prime power \(q\), we define
\begin{equation}
g_{k,\mu}^\Box (q)
=
\# \{ (t_1,\ldots,t_k) \in \CT_{(n)} (q)^k : t_1 \cdots t_k \in \CT_\mu^\Box (q) \}.
\end{equation}
We consider the quantity \(g_{k,\mu}^\Box (q)\) to also be a \(\GL_n \F_q\)-analogue of \(g_{k,\mu}\),
and we seek efficient ways to compute \(g_{k,\mu}^\Box (q)\).
The following is our first main result.

\begin{theorem}
\label{re_main}
For all \(n, k, \ell \in \N\) with \(n > 2\) and \(\ell > 1\), all prime powers \(q\), and all \(\mu = (\mu_1,\ldots,\mu_\ell) \vdash n\) with \(\mu_{\ell-1} > \mu_\ell = 1\), we have
\begin{equation}
\label{mu_ell_equals_1_part_eq}
g_{k,\mu}^\Box(q)
=
\frac{\# \CT_{(n)} (q)^k
\cdot
\# \CT_\mu^\Box (q)}{\# \GL_n \F_q}
\cdot
\sum_{r=0}^{n-1}
\frac{(-1)^{rk} \chi^{(n-r,1^r)}_\mu}{\left ( q^{\binom{r+1}{2}} \cdot {\genfrac[]{0pt}{1}{n - 1}{r}}_{q} \right )^{k-1}}.
\end{equation}
\end{theorem}

Compare \eqref{mu_ell_equals_1_part_eq} with the rephrasing
\begin{equation}
g_{k,\mu}
=
\frac{\# \CC_{(n)}^k \cdot \#\CC_\mu}{\# \FS_n}
\cdot
\sum_{r=0}^{n-1}
\frac{(-1)^{rk} \chi^{(n-r,1^r)}_\mu}{\binom{n-1}{r}^{k-1}}
\end{equation}
of \eqref{chi_phrasing_of_insp}.
Note, in particular, that, for the cases of \(\mu\) discussed in Theorem~\ref{re_main}, we have
\begin{equation}
\label{crazy_q_analogue_idea}
\lim_{q \to 1}
\frac{g_{k,\mu}^\Box (q) / \# \CT_{(n)} (q)^k}{\# \CT^\Box_\mu (q) / \# \GL_n \F_q}
=
\frac{g_{k,\mu} / \# \CC_{(n)}^k}{\# \CC_\mu / \# \FS_n}.
\end{equation}
Equation \eqref{crazy_q_analogue_idea} gives some justification that, after normalizing appropriately, \(g_{k,\mu}^\Box (q)\) is a \(q\)-analogue of \(g_{k,\mu}\) in the traditional \(q \to 1\) sense.

One special case of Theorem~\ref{re_main} is especially simple.
Note that \(g_{k,\mu} (q) = g_{k,\mu}^\Box (q)\) if all the parts of \(\mu\) are distinct.
\begin{corollary}
\label{n_minus_1_corollary}
For all \(n,k \in \N\) with \(n > 2\) and all prime powers \(q\), we have
\begin{equation}
\label{n_minus_1_equation}
g_{k,(n-1,1)} (q)
=
\frac{\# \CT_{(n)} (q)^k
\cdot
\# \CT_{(n-1,1)} (q)}{\# \GL_n \F_q}
\cdot
\left (
1 + \frac{(-1)^{nk-n-k}}{q^{\binom{n}{2}(k-1)}}
\right ).
\end{equation}
\end{corollary}

Compare \eqref{n_minus_1_equation}
to the analogous formula
\begin{equation}
\label{n_minus_1_phrasing_of_insp}
g_{k,(n-1,1)}
=
\frac{\# \CC_{(n)}^k
\cdot
\# \CC_{(n-1,1)}}{\# \FS_n}
\cdot
\left (
1 + (-1)^{nk-n-k}
\right )
\end{equation}
from the symmetric group.
Observe that \eqref{n_minus_1_phrasing_of_insp} is zero unless both \(n\) and \(k\) are even.
However, this behavior is not mimicked by \eqref{n_minus_1_equation}.

Our second main result is an explicit, albeit complicated, formula for \(g_{k,(n)} (q)\), which involves nested sums over divisors of \(n\).
We require some more notation before stating the result.
Throughout the paper, we make use of the standard \(q\)-analogues
\begin{align}
[m]_q  & = 1 + q + q^2 + \cdots + q^{m-1},
\label{x_analogue_of_n}\\
[m]_q! & = \prod_{\ell = 1}^{m} [\ell]_q, \quad \text{and}
\label{x_analogue_of_n_factorial}\\
{\genfrac[]{0pt}{1}{m}{\ell}}_q
& = \frac{[m]_q!}{[\ell]_q! [m-\ell]_q!},
\label{x_analogue_of_n_choose_k}
\end{align}
each of which is an integer polynomial in \(q\), for \(\ell, m \in \N\).
We denote the usual M\"obius function by \(\moebius\) to differentiate it from a partition named \(\mu\).
The rest of the necessary notation is contained in Table~\ref{formula_table} below.
Note that the \(\lcm\) in the denominator of \(C_{n,k,c}(q)\) from Table~\ref{formula_table} is computed in \(\Z\).

\begin{theorem}
\label{nu_n_part}
For all \(n, k \in \N\) and prime powers \(q\),
we have
\begin{equation}
\label{nu_n_part_eq}
g_{k,(n)}(q) = P_{n,k+1} (q)
\sum_{d | n}
(-1)^{n(k+1)/d}
d^{k}
D_{n,k+1,d} (q)
\sum_{c | d}
\moebius(d/c)
C_{n,k+1,c} (q),
\end{equation}
using the notation in Table~\ref{formula_table}.
\end{theorem}

The analogous formula from \(\FS_n\) is
\begin{equation}
\label{sns_n_version}
g_{k,(n)}
=
\frac{(n-1)!^{k}}{n}
\sum_{r=0}^{n-1}
\left (
\frac{(-1)^r}{\binom{n-1}{r}}
\right )^{k-1}.
\end{equation}
Unfortunately, it is not immediately obvious how to compare \eqref{nu_n_part_eq} with \eqref{sns_n_version}.

\begin{table}[ht]
\caption{Functions and their values for \(n \in \N\), \(d | n\), \(c | d\), and prime powers \(q\)}
\centering
\renewcommand{\arraystretch}{2}
\begin{tabular}{|c|c|}
\hline
\(f\) & \(f(q)\) \\
\hline
\(
\gamma_{n}\) & \(q^{\binom{n}{2}} (q-1)^n [n]_q!
\) \\
\(
P_{n,k}\) & \(
\frac{1}{\gamma_{n}(q)}
\left ( \frac{(-1)^n \gamma_n (q)}{n (q^n-1)} \right )^k
\) \\
\(
\deg_{n,d,r}\) & \(
q^{d \binom{r+1}{2}}
\cdot
\frac{\prod_{i=1}^n (q^i - 1)}{\prod_{j=1}^{n/d} (q^{jd} - 1)}
\cdot
{\genfrac[]{0pt}{1}{n/d - 1}{r}}_{q^d}
\) \\
\(
D_{n,k,d} \) & \(
\sum_{r = 0}^{\tfrac{n}{d} - 1}
(-1)^{rk} \deg_{n,d,r}(q)^{2-k}
\) \\
\(
C_{n,k,c}\) & \(
\sum_{s_1,\ldots,s_k | n}
\frac{(q^n-1)\prod_{i=1}^k \left [ (q^{s_i}-1) \moebius( n / s_i ) \right ]}{\lcm \left (
\tfrac{q^n-1}{q^c-1},
q^{s_1}-1,
\ldots,
q^{s_k}-1
\right ) }
\) \\
\hline
\end{tabular}
\label{formula_table}
\end{table}

\begin{remark}
\label{more_notation_and_rationality}
The polynomiality of \eqref{x_analogue_of_n}--\eqref{x_analogue_of_n_choose_k} implies that each of \(\gamma_n(q)\), \(P_{n,k}(q)\), \(\deg_{n,d,r}(q)\), and \(D_{n,k,d}(q)\) is rational in \(q\) with rational coefficients.
The only thing preventing \(g_{k,(n)} (q)\) from being rational function as well is that \(C_{n,k,c} (q)\) is not rational in general due to the fact that the \(\lcm\) function is not rational.
\end{remark}

\begin{remark}
There are many cases not addressed by Theorems \ref{re_main} and \ref{nu_n_part}.
One family of unaddressed cases is when \(m_1(\mu) > 1\).
Another is when \(\ell > 1\) and \(m_1(\mu) = 0\).
It is an open problem to find efficient formulas for \(g_{k,\mu} (q)\) in these cases.
\end{remark}

Our approach to proving Theorems~\ref{re_main} and \ref{nu_n_part} is to apply the same character-theoretic technique due to Frobenius that Stanley used in \cite{ncycle}.
In order to do so, we first prove the following result regarding the evaluation of \myemph{primary} characters of \(\GL_n \F_q\) on regular semisimple elements.
See Sections~\ref{two} and \ref{three} for missing notation.
In particular,
primary characters are denoted \(\chi^{f \mapsto \lambda}\), where \(f \in \F_q [z] \setminus \{z\}\) is monic, irreducible, and non-constant, and \(\lambda\) is a partition such that \(|\lambda| \cdot \deg f = n\).
Also note that \(\ell_f \in \Z\) and the codomain of the function \(\te\) is \(\C^\times\).

\begin{theorem} \label{main_lemma_in_intro}
Suppose \(n \in \N\), \(d | n\), \(\lambda \vdash n/d\), \(q\) is a prime power, \(f \in \CF_d (q)\), \(\mu \vdash n\), \(g \in \CT_\mu^\Box (q)\), and \(h_1,\ldots,h_{\ell(\mu)}\) are the distinct irreducible factors of the characteristic polynomial of \(g\).
If some part of \(\mu\) is not divisible by \(d\), then \(\chi^{f \mapsto \lambda} (g) = 0\).
Otherwise, there exists \(\tilde{\mu} \vdash n/d\) such that \(\mu = d \tilde{\mu}\), and
\begin{equation}
\chi^{f \mapsto \lambda} (g)
= (-1)^{\tfrac{n}{d}(d-1)}
\chi^\lambda_{\tilde{\mu}}
\prod_{i = 1}^{\ell(\mu)}
\frac{1}{\tilde{\mu}_i}
\sum_{\substack{ \beta_i \in \F_{q^{\mu_i}} \\ h_i(\beta_i) = 0}} \te (\beta_i)^{\ell_f [\tilde{\mu}_i]_{q^d}}.
\end{equation}
\end{theorem}

Theorem~\ref{main_lemma_in_intro} also enables us to answer probabilistic questions about multiplying regular elliptic elements randomly.
We define
\begin{equation}
p_{k,\mu} (q) = \frac{g_{k,\mu}(q)}{\# \CT_{(n)} (q)^k},
\end{equation}
which is the probability that the product of a randomly chosen \(k\)-tuple of regular elliptic elements is in \(\CT_\mu (q)\).
We are only concerned with the nontrivial cases \(k \geq 2\).
Of course, Theorems~\ref{re_main} and \ref{nu_n_part} provide exact formulas for \(p_{k,\mu} (q)\) in certain special cases, but we are also interested in the behavior of \(p_{k,\mu}(q)\) as \(q\) becomes arbitrarily large.
For the case of regular semisimple elements,
we define
\begin{equation}
\label{pBox_defs}
p_{k,\mu}^\Box (q) = \frac{g_{k,\mu}^\Box(q)}{\# \CT_{(n)} (q)^k}
\end{equation}
Again, Theorems~\ref{re_main} and \ref{nu_n_part} provide exact formulas for \(p^\Box_{k,\mu} (q)\) in some special cases.
However, we are able to compute \(\lim_{q \to \infty} p_{k,\mu} (q)\) and \(\lim_{q \to \infty} p_{k,\mu}^\Box (q)\) for all \(\mu \vdash n\).

\begin{theorem}
\label{intro_version_p}
For all \(n,k \in \N\) with \(k \geq 2\) and \(\mu \vdash n\), we have
\begin{equation}
\lim_{q \to \infty} p_{k,\mu}      (q) =
\lim_{q \to \infty} p_{k,\mu}^\Box (q) =
\frac{1}{z_\mu}.
\end{equation}
\end{theorem}

According to Corollary~\ref{dense}, one interpretation of Theorem~\ref{intro_version_p} is that, for large \(q\), random products of regular elliptic elements are approximately distributed uniformly throughout \(\GL_n \F_q\).
We do not currently have a heuristic explanation for this behavior, nor do we know how random products of regular elliptic elements are distributed among individual conjugacy classes.

Even though Theorem~\ref{intro_version_p} describes the asymptotics of \(g_{k,\mu}(q)\) as \(q \to \infty\), it does not address the specific behavior of \(g_{k,\mu}(q)\) for small \(q\).
It turns out that Theorem~\ref{re_main} gives a family of examples where the function \(g_{k,\mu}^\Box (q)\) is a polynomial in \(q\).
See Corollary~\ref{full_polynomiality_for_some}.
However, as discussed in Remark~\ref{more_notation_and_rationality} above, \(g_{k,(n)} (q)\) is not necessarily a polynomial, or even a rational function, of \(q\).
Instead, we have the following result.

\begin{corollary}[to~Thm.~\ref{nu_n_part}]
\label{polynomiality_for_nu_n}
Suppose \(n,k \in \N\) and \(n\) is prime.
There exist polynomials \(f_0, f_1, \ldots, f_{n-1} \in \Q[x]\) with the property that, for each \(i \in \{0,\ldots,n-1\}\), we have
\begin{equation}
\label{its_a_quasipolynomial_yay}
g_{k,(n)}(q) = f_i(q) \quad \text{for all prime powers } q \equiv i \pmod n.
\end{equation}
In other words, \(g_{k,(n)} (q)\) is a quasipolynomial in \(q\) of quasiperiod \(n\).
\end{corollary}

The rest of the paper is organized as follows.
In Section~\ref{two}, we discuss some preliminary information, including the character-theoretic technique and details regarding the symmetric groups, finite fields, and the finite general linear groups.
In Section~\ref{three}, we provide a concise retelling of Green's original formulation of the characters of the finite general linear groups \cite{green}.
In Section~\ref{main_results_section}, we prove Theorems~\ref{re_main} and \ref{nu_n_part}, our main enumerative results.
In Section~\ref{probabilistic_section}, we prove Theorem~\ref{intro_version_p}, our main probabilistic result.
In Section~\ref{outro}, we discuss polynomiality, prove Corollary~\ref{polynomiality_for_nu_n}, and list some open problems.

\section{Preliminaries}
\label{two}

\subsection{The character theory approach}
\label{approach}

We assume basic knowledge of the ordinary complex character theory of finite groups.
All characters we refer to in this paper are complex.
See Fulton-Harris \cite{FH} or Serre \cite{serre} for a review.
We will make use of a standard character-theoretic technique based on the following result due to Frobenius.
Let \(G\) be a finite group, and let \(\Irr (G)\) denote the set of all irreducible characters of \(G\).
For \(\chi \in \Irr (G)\), let \(\deg \chi\) denote the value of \(\chi\) at the identity element of \(G\).
For a straightforward proof of the following theorem, see Zagier's Appendix A in Lando-Zvonkin\cite{LZ}.

\begin{theorem}[Frobenius \cite{frob}]\label{frob}
Let \(k\) be a positive integer, and, for each \(i \in \{1,\ldots,k\}\), let \(A_i\) be a union of conjugacy classes in \(G\).
For any \(g \in G\), the number of tuples \((t_1,\ldots,t_k) \in A_1 \times \cdots \times A_k\) such that \(t_1 \cdots t_k = g\) is given by
\begin{equation}
\frac{1}{\# G} \sum_{\chi \in \Irr (G)} (\deg \chi)^{1-k} \chi (g^{-1}) \prod_{i = 1}^k \sum_{t \in A_i} \chi(t).
\label{ff}
\end{equation}
\end{theorem}

\begin{corollary}
\label{ff_in_situ_corollary}
For all \(n,k \in \N\), \(\mu \vdash n\), and prime powers \(q\),
\begin{equation}
g_{k,\mu} (q)
=
\label{ff_in_situ}
\frac{1}{\# \GL_n \F_q}
\sum_{\chi \in \Irr (\GL_n \F_q)}
(\deg \chi)^{1-k}
\left ( \sum_{g \in \CT_{(n)} (q)} \chi(g) \right )^k
\left ( \sum_{h \in \CT_\mu (q)}   \chi(h) \right ).
\end{equation}
Moreover, the same is true when both \(g_{k,\mu} (q)\) is replaced with \(g_{k,\mu}^\Box (q)\) and \(\CT_\mu (q)\) is replaced with \(\CT_\mu^\Box (q)\).
\end{corollary}

\begin{proof}
Consider applying Theorem~\ref{frob} to the case of \(k+1\) factors, the first \(k\) of which are regular elliptic and the last of which has cycle type \(\mu\).
Each \(\CT_\mu (q)\) is a union of conjugacy classes, and so the hypotheses of Theorem~\ref{frob} are satisfied.
Moreover, each \(\CT_\mu (q)\) is closed under taking inverses, implying factorizations of the form
\[
(t_1,\ldots,t_k) \in \CT_{(n)}(q)^k
\quad
\text{such that}
\quad
t_1 \cdots t_{k+1} \in \CT_\mu (q)
\]
are in bijection with factorizations of the form
\[
(t_1,\ldots,t_k,t_{k+1}) \in \CT_{(n)}(q)^k \times \CT_\mu (q)
\quad
\text{such that}
\quad
t_1 \cdots t_{k+1} = \id.
\]
Thus,
\[
g_{k,\mu} (q)
=
\# \{ (t_1,\ldots,t_k,t_{k+1}) \in \CT_{(n)} (q)^k \times \CT_\mu (q) : t_1 \cdots t_{k+1} = \id \}.
\]
Applying Theorem~\ref{frob} with
\[
A_1 = A_2 = \cdots = A_k = \CT_{(n)}(q), \quad A_{k+1} (q) = \CT_\mu (q), \quad \text{and } g = \id
\]
gives the result.
The final claim follows from the same argument.
\end{proof}

\subsection{The symmetric groups and partitions}

We will require some specific information about the irreducible characters of the symmetric group \(\FS_n\).
This information can be found in Stanley \cite{EC2} and Sagan \cite{sagan}. See also Fulton-Harris \cite{FH} or Fulton \cite{fulton} for added discussion on the irreducible characters of the symmetric groups.

A \myemph{partition} of \(n\) is a weakly decreasing sequence of non-negative integers \(\mu = (\mu_1,\mu_2,\ldots)\) such that \(\sum_{i \geq 1} \mu_i = n\), denoted by \(\mu \vdash n\).
Denote by \(\emptyset\) the unique partition of \(0\) and by \(\Box\) the unique partition of \(1\).
Let \(\Par\) denote the set of all partitions of all non-negative integers.
Each \(\mu_i\) is called a \myemph{part} of \(\mu\).
The number of nonzero parts of \(\mu\) is called the \myemph{length} of \(\mu\) and is denoted by \(\ell(\mu)\).
The \myemph{conjugate} of \(\mu\) is denoted by \(\mu'\) and defined by \(\mu'_i = \# \{j \geq 1 : \mu_j \geq i\}\) for all \(i \geq 0\).
The \myemph{multiplicity} of a positive integer \(i\) in a partition \(\mu\) is defined as \(\# \{j \geq 1 : \mu_j = i \}\) and denoted by \(m_i (\mu)\).
Also define \(s_i (\mu) = \sum_{j=1}^i \mu_j \).
In general, if some part of a partition is repeated, we denote this with a superscript.
Moreover, we omit zeros.
For example, \((3,2^4)\) is the same as \((3,2,2,2,2)\).
A partition of the form \((n-r,1^r) \vdash n\) for some \(r \in \{0,\ldots,n-1\}\) is called a \myemph{hook}.
An important statistic on partitions is \(\mu \mapsto z_\mu\) defined by \(z_\mu = \prod_{i \geq 1} i^{m_i(\mu)} \cdot m_i(\mu)!\).
If \(d\) is a positive integer, then we use \(d \mu\) to denote the partition \((d \mu_1, d \mu_2, \ldots )\) of \(dn\).

The conjugacy classes of \(\FS_n\) are in bijection with the partitions of \(n\) as follows.
If \(\pi = (\pi_{1,1},\ldots,\pi_{1,\mu_1}) \cdots (\pi_{\ell,1},\ldots,\pi_{\ell,\mu_\ell}) \in \FS_n\) is a cycle decomposition of $\pi$ with $\mu_1 \geq \mu_2 \geq \cdots \geq \mu_\ell$, then the conjugacy class of $\pi$ is indexed by the partition $\mu = (\mu_1,\mu_2,\ldots,\mu_\ell)$.
The partition \(\mu\) is called the \myemph{cycle type} of the permutation $\pi$.
Let $\CC_\mu$ denote the conjugacy class consisting of those permutations with cycle type $\mu$.
The statistic $\mu \mapsto z_\mu$ has the following algebraic interpretation.
If $\sigma \in \FS_n$ has cycle type $\mu$, then $z_\mu$ is the number of permutations in $\FS_n$ which commute with $\sigma$.
By the orbit-stabilizer theorem \cite[Prop.~4.3.6]{DF},
\(\# \CC_\mu = n! / z_\mu\).

The irreducible characters of \(\FS_n\) are indexed by partitions of \(n\) in a standard way.
If $\lambda \vdash n$, let $\chi^\lambda$ denote the character indexed by $\lambda$.
Let $\chi^\lambda_\mu$ denote $\chi^\lambda$ evaluated on any element of \(\CC_\mu\).
There is a combinatorial formula, known as the Murnaghan-Nakayama (MN) rule, for computing the irreducible character values for the symmetric groups \cite{murnaghan, nakayama}. See \cite[Thm.~7.17.3]{EC2} for a full statement and proof.
We will use the following two special cases.

\begin{samepage}
\begin{corollary}[to the MN rule] \label{MNrule_ncycle}
For all \(n \in \N\) and \(\lambda \vdash n\), we have
\begin{equation}
\chi^\lambda_{(n)} =
\begin{cases}
(-1)^r, & \lambda = (n-r,1^r) \text{ for some } r \in \{0,\ldots,n-1\}, \\
0, & \text{otherwise},
\end{cases}
\end{equation}
and
\begin{equation}
\chi^\lambda_\mu = \begin{cases}
1,          & \lambda = (n), \\
(-1)^n,     & \lambda = (1^n), \\
(-1)^{r-1}, & \lambda = (n-r, 2, 1^{r-2}) \text{ for some } r \in \{2,\ldots,n-2\}, \\
0,          & \text{otherwise}.
\end{cases}
\end{equation}
\end{corollary}
\end{samepage}

\subsection{Finite fields}

We assume some basic knowledge about finite fields, all of which can be found in Dummit-Foote \cite{DF}.
For all positive integers \(m\), there is a degree \(m\) field extension \(\F_{q^m}\) of \(\F_q\).
For positive integers \(m\) and \(m'\), we have the containment \(\F_{q^m} \subset \F_{q^{m'}}\) if and only if \(m | m'\).
For any field \(\F\), let \(\F^\times\) denote the multiplicative group of its nonzero elements, called the \myemph{unit group}.
The unit group of any finite field is cyclic.
Moreover, in the case \(m | m'\), we have that \(\F_{q^m}^\times\) is a subgroup of \(\F_{q^{m'}}^\times\).

Let \(\CF (q) \subset \F_q[z] \) denote the set of monic, nonconstant, irreducible polynomials over \(\F_q\), excluding \(z\) itself.
For each \(d \in \N\), let \(\CF_d (q) = \{f \in \CF (q) : \deg f = d\}\).
Let \(\sqcup\) denote disjoint union.
\begin{lemma}
\label{root_union}
For all \(d \in \N\) and prime powers \(q\),
\begin{equation}
\label{root_union_eq}
\F_{q^d}^\times
=
\bigsqcup_{c | d}
\bigsqcup_{f \in \CF_c (q)}
\{\alpha \in \F_{q^c}^\times : f(\alpha) = 0\}.
\end{equation}
\end{lemma}
\begin{proof}
Every element of \(\F_{q^d}^\times\) is the root of an element of \(\CF (q)\) with degree dividing \(d\).
The union is disjoint because distinct monic, irreducible polynomials over \(\F_q\) do not have shared roots.
\end{proof}

For each \(n \in \N\), fix a generator \(\eps\) of the cyclic group \(\F_{q^{n!}}^\times\), and fix an injective group homomorphism \(\te : \F_{q^{n!}}^\times \to \C^\times\) mapping \(\eps \mapsto e^{2 \pi i / (q^{n!}-1)}\).
Note that we omit the dependence of \(\eps\) on \(n\).
Context will suffice.
For each \(d \in \{1,\ldots,n\}\), let \(\eps_d\) denote \(\eps\) raised to the power \( (q^{n!}-1) / (q^d-1)\).
The multiplicative order of \(\eps_d\) is \(q^d-1\), and \(\eps_d\) is a cyclic generator of \(\F_{q^d}^\times\).
Also, \(\te\) maps \(\F_{q^d}^\times\) isomorphically onto the group of \((q^d-1)^\text{th}\) roots of unity.

\begin{corollary}
\label{theta_of_root_union}
For all \(n \in \N\), \(d \in \{1,\ldots,n\}\), and prime powers \(q\),
\begin{equation}
\label{theta_of_root_union_eq}
\{\xi \in \C^\times : \xi^{q^d-1} = 1\}
=
\bigsqcup_{c | d}
\bigsqcup_{f \in \CF_c (q)}
\{\te(\alpha) : \, \alpha \in \F_{q^c}^\times, \, f(\alpha) = 0\}.
\end{equation}
\end{corollary}

\begin{proof}
Apply \(\te\) to each element on the left and right sides of \eqref{root_union_eq}.
\end{proof}

For each \(d \in \N\), the Galois group of \(\F_{q^d}\) over \(\F_q\) is cyclic of order \(d\), generated by the field automorphism
\[
\F_{q^d} \to \F_{q^d}, \quad \alpha \mapsto \alpha^q.
\]
Therefore, for each \(f \in \CF_d (q)\), if \(\alpha\) is any root of \(f\), then
\(\alpha,\alpha^q,\alpha^{q^2},\ldots,\alpha^{q^{d-1}}\)
are distinct and are all the roots of \(f\).
Since \(\F_{q^d}^\times\) is generated by \(\eps_d\), there exists some \(\ell \in \Z\) such that \(\eps_d^\ell\) is a root of \(f\).
Assign to \(f\) an arbitrary integer \(\ell_f\) such that \(\eps_d^{\ell_f}\) is a root of \(f\).
To combine the previous three sentences,
\begin{equation}
\label{who_are_my_roots}
f(z) = \prod_{i=0}^{d-1}
\left (
z - \left ( \eps_d^{\ell_f} \right )^{q^i}
\right ).
\end{equation}
Observe that the choice of \(\ell_f\) is unique up to multiplication by powers of \(q\) and addition by multiples of \(q^d - 1\).
Our results are independent of the choice of \(\ell_f\).

In case we are considering a polynomial \(f\) with degree \(d\) dividing \(n\), we will also make use of the quantity \(\ell_f [n/d]_{q^d}\), viewing it as an element of \(\Z / (q^n-1)\).
More precisely, define the group isomorphism
\begin{equation}
\label{te_n_def}
\te_n : \F_{q^n}^\times \to \Z / (q^n-1)
\quad \text{by} \quad
\te_n \left ( \eps_n^{\ell} \right ) = \ell \mod q^n-1
\quad
\forall \ell \in \Z.
\end{equation}
It follows that
\(\te_n\) maps \(\eps_d^{\ell}\) to \(\ell [n/d]_{q^d}\).

\begin{corollary}[to~Lem.~\ref{root_union}]
\label{te_n_map_property}
For all \(n \in \N\), \(d | n\), and prime powers \(q\),
\begin{equation}
\{ m \cdot [n/d]_{q^d} : m \in \Z / (q^n-1) \}
=
\bigsqcup_{c | d}
\bigsqcup_{f \in \CF_c (q)}
\{\te_n(\alpha) : f(\alpha) = 0\}.
\end{equation}
\end{corollary}

\begin{proof}
Apply \(\te_n\) to each element on the left and ride sides of \eqref{root_union_eq}, recalling that \(d | n\) and \(\F_{q^d}^\times\) is the unique subgroup of \(\F_{q^n}^\times\) of order \(q^d-1\).
\end{proof}

\begin{example}
Consider the case \(q = 3\), \(n = 4\), and \(\te : \eps \mapsto \zeta\), where
\[\zeta = e^{2 \pi i / (q^{n!}-1)}.\]
We write \(\F_3\) as \(\{0,1,2\}\) under addition and multiplication modulo \(3\).
We record in Table~\ref{root_table} the polynomials \(f \in \CF (q)\) with degree dividing \(n\), together with all possible choices for \(\ell_f\) modulo \(q^d-1\) and all possible choices for \(\ell_f [n / \deg f]_{q^{\deg f}}\) module \(q^n-1\).
For the sake of brevity, we omit most of the degree four polynomials.
Note that these values depend on the choice of \(\eps\).
\begin{table}[ht]
\caption{Choices for \(\ell_f\) and \(\ell_f [n/\deg f]_{q^{\deg f}}\) with \(q = 3\) and \(n = 4\).}
\centering
\begin{tabular}{r|c|c}
\(f\) & \(\ell_f\) & \(\ell_f [n/\deg f]_{q^{\deg f}}\) \\ \hline \hline
\(z + 2\) &                  \( 0 \) & \(0\) \\
\(z + 1\) &                  \( 1 \) & \(40\) \\
\hline
\(z^2 + 2z + 2\) &           \( 1, 3 \) & \( 10, 30 \) \\
\(z^2 + 1\) &                \( 2, 6 \) & \( 20, 60 \) \\
\(z^2 + z + 2\) &            \( 5, 7 \) & \( 50, 70 \) \\
\hline
\(z^4 + 2z^3 + 2\) &         \( 1,  3,  9, 27 \) & \( 1,  3,  9, 27 \) \\
\(z^4 + 2z^3 + z^2 + 1\) &   \( 2,  6, 18, 54 \) & \( 2,  6, 18, 54 \) \\
\(z^4 + z^3 + 2z + 1\) &     \( 4, 12, 28, 36 \) & \( 4, 12, 28, 36 \) \\
\vdots & \vdots & \vdots
\end{tabular}
\label{root_table}
\end{table}

We can also visualize the data from Table~\ref{root_table} in the complex plane as follows.
Observe that \(\te\) maps \(\eps_n\) to
\[
\xi = \zeta^{\frac{q^{n!}-1}{q^n-1}},
\]
a \((q^n-1)^\text{th} = 80^\text{th}\) root of unity.
Given a choice of \(\ell_f\) for some \(f\) in Table~\ref{root_table} with degree \(d | n\), we have
\[
\alpha
=
\eps_d^{\ell_f}
=
\eps_n^{\ell_f [n/d]_{q^d}}
\]
is a root of \(f\),
\[
\te (\alpha) = \xi^{\ell_f [n/d]_{q^d}} \in \C^\times,
\quad
\text{and}
\quad
\te_n (\alpha) = \ell_f [n/d]_{q^d} \in \Z / (q^n-1).
\]
Figure~\ref{root_picture} shows the complex plane and the images under \(\te\) of all the roots of all the polynomials \(f \in \CF_1 (3) \sqcup \CF_2 (3) \sqcup \CF_4 (3)\).
The images under \(\te\) of roots of polynomials in \(\CF_1 (3)\) are labeled by the largest nodes, those for \(\CF_2 (3)\) by the medium-sized nodes, and those for \(\CF_4 (3)\) by the smallest nodes.

One can observe in Figure~\ref{root_picture} the following instance of Cor.~\ref{theta_of_root_union}.
The images under \(\te\) of the roots of polynomials in \(\CF_1 (3) \sqcup \CF_2 (3)\) precisely form the set of \((q^2-1)^\text{th} = 8^\text{th}\) roots of unity, pictorially represented by the medium and large nodes.

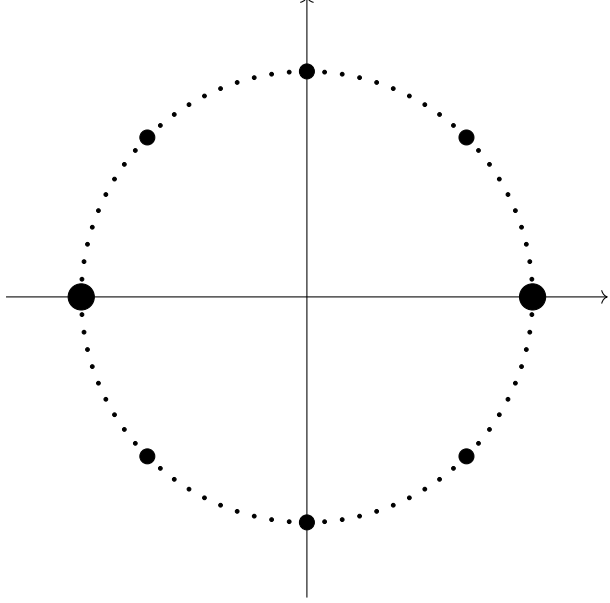
\begin{figure}
\caption{Images under \(\te\) of roots of polynomials from Table~\ref{root_table}}
\centering
\begin{tikzpicture}[
  dot4/.style={draw,fill,circle, inner sep = 0.5pt},
  dot2/.style={draw,fill,circle, inner sep = 2pt},
  dot1/.style={draw,fill,circle,inner sep = 3.5pt}
  ]
  \draw[->] (-4,0) -- (4,0);
  \draw[->] (0,-4) -- (0,4);

  \foreach \i in {0,40} {
    \node[dot1] at (\i*4.5:3) {} ;
  }
  \foreach \i in {10,20,30,50,60,70} {
    \node[dot2] at (\i*4.5:3) {} ;
  }
  \foreach \i in {1,2,3,4,5,6,7,8,9,11,12,13,14,15,16,17,18,19,
                  21,22,23,24,25,26,27,28,29,31,32,33,34,35,36,37,38,39,
                  41,42,43,44,45,46,47,48,49,51,52,53,54,55,56,57,58,59,
                  61,62,63,64,65,66,67,68,69,71,72,73,74,75,76,77,78,79} {
    \node[dot4] at (\i*4.5:3) {} ;
  }

\end{tikzpicture}
\label{root_picture}
\end{figure}

One can also observe the following instance of Cor.~\ref{te_n_map_property} in either Table~\ref{root_table} or Fig.~\ref{root_picture}.
The images under \(\te_n\) of the roots of polynomials in \(\CF_1 (3) \sqcup \CF_2 (3)\) are \(0, 10, 20, \ldots, 70\).
These are precisely the multiples of \((q^n-1)/(q^2-1) = 10\) in \(\Z / (q^n-1)\).
\end{example}

\subsection{The finite general linear groups}

The finite general linear group \(\GL_n \F_q\) is the group of \(n \times n\) invertible matrices with entries in the finite field \(\F_q\) with \(q\) elements.
We will occasionally have need to view the elements of \(\GL_n \F_q\) abstractly as linear transformations on an \(n\)-dimensional \(\F_q\)-vector space.
The cardinality of \(\GL_n \F_q\) is \(\gamma_n(q)\) as defined in Table~\ref{formula_table} \cite[Prop.~1.10.1]{EC1}.

\subsubsection{Indexing the \(\GL_n \F_q\) conjugacy classes} \label{indexing}

We first discuss how to index the conjugacy classes and irreducible characters of $\GL_n \F_q$.
Let $V = \F_q^n$.
Then $\GL_n \F_q$ acts on $V$ via matrix multiplication.

Consider a fixed $g \in \GL_n \F_q$.
Let \(V_g\) denote the vector space $V$ endowed with an $\F_q [z]$-module structure by defining the action $\F_q [z] \times V_g \to V_g$ to be $(f(z), v) \mapsto f(g)(v)$.
The polynomial ring $\F_q[z]$ is a principal ideal domain, and $V$ is finite-dimensional as an $\F_q$-vector space, hence $V_g$ is finitely generated as an $\F_q[z]$-module.
By the structure theorem for finitely generated modules over principal ideal domains \cite[Thm.~12.1.6]{DF}, there exists a unique function $\ulambda^g : \CF (q) \to \Par$ such that
\begin{equation} \label{rcf}
V_g \isom \bigoplus_{f \in \CF (q)} \bigoplus_{i \geq 1} \F_q [z] \left / \left ( f(z)^{\ulambda^g (f)_i} \right ) \right .
\end{equation}
as $\F_q [z]$-modules,
where $\ulambda^g (f)_i$ denotes the $i^\text{th}$ part of $\ulambda ^g (f)$.
We say that \(g\) \myemph{determines the isomorphism} \eqref{rcf}.
Moreover, \(g_1\) and \(g_2\) are conjugate in \(\GL_n \F_q\) if and only if \(\ulambda^{g_1} = \ulambda^{g_2}\).
If \(g\) is clear from context, we omit the superscript from \(\ulambda^g\).
The function $\ulambda : \CF (q) \to \Par$ is said to \myemph{index} the conjugacy class of $g \in \GL_n \F_q$,
and we denote this conjugacy class by \(\CC_\ulambda\).

Define the \myemph{norm} of an index \(\ulambda : \CF (q) \to \Par\) to be
\begin{equation}
\| \ulambda \| = \sum_{f \in \CF (q)} |\ulambda (f)| \cdot \deg f.
\end{equation}
Computing dimensions of each side in the isomorphism given in \eqref{rcf} implies the equation $n = \| \ulambda \|$.
Therefore, to each conjugacy class $C \subset \GL_n \F_q$, we can associate a unique index $\ulambda$ with $n = \| \ulambda \|$ such that \(C = \CC_\ulambda\).
In \cite{green}, Green shows that the condition $n = \| \ulambda \|$ is necessary and sufficient for $\ulambda$ to be the index of some conjugacy class in $\GL_n \F_q$.
Thus, conversely, to every index $\ulambda$ with $\| \ulambda \| = n$, there exists a unique conjugacy class of $\GL_n \F_q$ with index $\ulambda$.

Given $\ulambda$, one can read off the characteristic and minimal polynomials of $g$ as follows.
The minimal polynomial is
$\prod_{f \in \CF (q)} f^{\ulambda(f)_1}$,
and the characteristic polynomial is
$\prod_{f \in \CF (q)} f^{|\ulambda (f) |}$.
Moreover, we can use the isomorphism \eqref{rcf} to write down a specific matrix whose conjugacy class is indexed by $\ulambda$ as follows.
If $h(z) = z^n - a_{n-1}z^{n-1} - \cdots - a_1 z - a_0 \in \F_q[z]$, then the \myemph{companion matrix} of $h$ is defined by
$$
A(h) = \begin{bmatrix}
0 & 0 & \cdots & 0 & a_0    \\
1 & 0 & \cdots & 0 & a_1    \\
0 & 1 & \cdots & 0 & a_2    \\
\vdots  &   & \ddots &   & \vdots \\
0 & 0 & \cdots & 1 & a_{n-1}
\end{bmatrix},
$$
where $A(1)$ is the empty matrix.
Given $g \in \GL_n \F_q$ which determines the isomorphism \eqref{rcf}, $g$ is in the same conjugacy class as any block-diagonal matrix
whose diagonal blocks are those nonempty matrices $A(f^{\ulambda(f)_i})$ for $f \in \CF (q)$ and $i \geq 1$.
Any block-diagonal matrix with these diagonal blocks arranged in any order of non-increasing size from the upper-left corner to the lower-right corner is said to be a \myemph{rational canonical form} of $g$.
Furthermore, if $g$ itself is in this form, say that \myemph{$g$ is in rational canonical form}.

\begin{example}
Consider the matrix
$$ g = \begin{bmatrix}
0 & 1 & 0 \\
1 & 1 & 0 \\
0 & 0 & 1
\end{bmatrix} \in \GL_3 \F_2. $$
This matrix is block-diagonal, with diagonal blocks of size $2$ and $1$. The blocks are the companion matrices of the polynomials $z^2+z+1$ and $z+1 \in \F_2[z]$, respectively.
Thus, $g$ is in rational canonical form. The conjugacy class of $g$ has index $\ulambda$ defined by
$$ \ulambda (f) = \begin{cases}
(1) & \text{if } f = z^2 + z + 1, \\
(1) & \text{if } f = z+1, \\
\emptyset & \text{otherwise.}
\end{cases} $$
The characteristic polynomial of $g$ is $(z+1)(z^2+z+1) = z^3+1 \in \F_2 [z]$, which is also its minimal polynomial.
\end{example}

Define the \myemph{support} of an index $\ulambda$ by $\supp \ulambda = \{f \in \CF (q) : \ulambda (f) \neq \emptyset\}$. Observe $\| \ulambda \| < \infty$ implies $\# \supp \ulambda  < \infty$. Call an index $\ulambda$ \myemph{primary} if $\# \supp \ulambda = 1$.
If $\ulambda$ is primary with $\supp \ulambda = \{f\}$ and $\ulambda (f) = \lambda$, then we denote $\ulambda$ simply by $f \mapsto \lambda$. For example, $z-1 \mapsto (1^n)$ is the index for the identity matrix of $\GL_n \F_q$.
We refer to a conjugacy class itself as primary if its index is primary, and we refer to an element as primary if it is a member of a primary conjugacy class.

The following result allows us to compute the sizes of conjugacy classes in \(\GL_n \F_q\).
Recall that, for \(\mu \vdash n\) and \(i \in \N\), we have defined \(s_i(\mu) = \sum_{j=1}^i \mu_j\).

\begin{theorem}[{\cite[Thm.~1.10.7]{EC1}}]
\label{cc_sizes}
For all \(n \in \N\), prime powers \(q\), and \(\ulambda : \CF (q) \to \Par\) with \(\|\ulambda \| = n\), we have
\begin{equation}
\# \CC_\ulambda
=
\frac{\gamma_n(q)}{
\prod_{f \in \CF (q)}
\prod_{i \geq 1}
\prod_{j = 1}^{m_i(\ulambda(f))}
\left (
(q^{\deg f})^{s_i(\ulambda(f)')}
-
(q^{\deg f})^{s_i(\ulambda(f)') - j}
\right ) }.
\end{equation}
\end{theorem}

\begin{example}
\label{conj_class_example}
Consider $\GL_3 \F_2$.
The degree 1, 2, and 3 polynomials in $\CF (2)$ are
$f_1 = z+1, f_2 = z^2+z+1, f_3 = z^3+z^2+1$, and $\tilde{f}_3 = z^3+z+1$.
There are six functions $\ulambda : \CF (2) \to \Par$ satisfying $\| \ulambda \| = 3$, which index the conjugacy classes and irreducible characters of $\GL_3 \F_2$. The primary ones are
$$
f_1 \mapsto (1,1,1), \,
f_1 \mapsto (2,1), \,
f_1 \mapsto (3), \,
f_3 \mapsto (1), \,
\text{ and } \,
\tilde{f}_3 \mapsto (1).
$$
There is only one more left to define. We call it $\ulambda_0$. It is defined by
$$ \ulambda_0 (f) =
\begin{cases}
(1) & \text{if } f = f_1, \\
(1) & \text{if } f = f_2, \\
\emptyset & \text{otherwise.}
\end{cases}
$$
We now name all of the conjugacy classes, indicate a member in rational canonical form, and indicate what function $\CF (2) \to \Par$ indexes the class.
\begin{align*}
\text{The conjugacy class } U_{1} \text{ of }
\begin{bsmallmatrix}
1 & 0 & 0 \\
0 & 1 & 0 \\
0 & 0 & 1
\end{bsmallmatrix}
&
\text{ is indexed by }
f_1 \mapsto (1,1,1). \\
\text{The conjugacy class } U_{2} \text{ of }
\begin{bsmallmatrix}
0 & 1 & 0 \\
1 & 0 & 0 \\
0 & 0 & 1
\end{bsmallmatrix}
&
\text{ is indexed by }
f_1 \mapsto (2,1). \\
\text{The conjugacy class } U_{3} \text{ of }
\begin{bsmallmatrix}
0 & 0 & 1 \\
1 & 0 & 1 \\
0 & 1 & 1
\end{bsmallmatrix}
&
\text{ is indexed by }
f_1 \mapsto (3). \\
\text{The conjugacy class } E \text{ of }
\begin{bsmallmatrix}
0 & 0 & 1 \\
1 & 0 & 0 \\
0 & 1 & 1
\end{bsmallmatrix}
&
\text{ is indexed by }
f_3 \mapsto (1). \\
\text{The conjugacy class } \tilde{E} \text{ of }
\begin{bsmallmatrix}
0 & 0 & 1 \\
1 & 0 & 1 \\
0 & 1 & 0
\end{bsmallmatrix}
&
\text{ is indexed by }
\tilde{f}_3 \mapsto (1). \\
\text{The conjugacy class } C_{0} \text{ of }
\begin{bsmallmatrix}
0 & 1 & 0 \\
1 & 1 & 0 \\
0 & 0 & 1
\end{bsmallmatrix}
&
\text{ is indexed by }
\ulambda_0. \\
\end{align*}
We chose these names for the following reasons. The \textbf{u}nipotent classes are $U_1, U_2$, and $U_3$. The regular \textbf{e}lliptic classes are $E$ and $\tilde{E}$. The \textbf{o}dd \textbf{o}ne \textbf{o}ut is $C_0$.
As an example of Theorem~\ref{cc_sizes}, we compute the cardinality of \(U_2\).
Recall that the index for \(U_2\) is primary with support \(\{f_1\}\).
Furthermore, the image of \(f_1\) is \((2,1) = (2,1)' \vdash 3\), and \(m_1((2,1)) = m_2((2,1)) = 1\).
By Theorem~\ref{cc_sizes},
\begin{equation}
\# U_2
=
\frac{\gamma_3(2)}{
\prod_{i = 1}^2
\left (
2^{s_i((2,1))}
-
2^{s_i((2,1)) - 1}
\right )
}
=
\frac{2^{\binom{3}{2}} [3]_2 !}{(2^2 - 2)(2^3 - 2^2)}
=
21.
\end{equation}

\end{example}

\subsubsection{Cycle type, regular semisimple elements, and regular elliptic elements}
\label{cycle_rs_re}

Recall the definition of \myemph{cycle type} for \(\GL_n \F_q\) from the introduction.
The definition given in the introduction is equivalent to the following.
For any matrix \(g \in \GL_n \F_q\), \(\type (g) = \mu\) if and only if
\[
m_i (\mu)
=
\sum_{f \in \CF_i (q)}
| \ulambda^g (f) |
\]
for each \(i \in \{1,\ldots,n\}\).
Recall that we define, for \(\mu \vdash n\) and \(q\) a prime power,
\begin{equation}
\label{type_set}
\CT_\mu (q) = \{g \in \GL_n \F_q : \type (g) = \mu \}.
\end{equation}
Since conjugate matrices have the same characteristic polynomial, each $\CT_\mu (q)$ is a union of conjugacy classes, and $\{\CT_\mu (q): \mu \vdash n\}$ forms a partition of $\GL_n \F_q$.
\begin{example}
Using the notation from Ex.~\ref{conj_class_example} above,
\(\CT_{(1,1,1)} (2) = U_1 \cup U_2 \cup U_3\),
\(\CT_{(2,1)} (2) = C_0\), and
\(\CT_{(3)} (2) = E \cup \tilde{E}\).
\end{example}

We now discuss a special class of matrices in $\GL_n \F_q$, the \myemph{regular semisimple} elements.
An element of an algebraic group is called \myemph{regular} if the dimension of its centralizer is equal to the dimension of a maximal torus of the group. A matrix in $\GL_n \F_q$ is called \myemph{semisimple} if it is diagonalizable over an algebraic closure of $\F_q$. A matrix in $\GL_n \F_q$ is called \myemph{regular semisimple} if it is both regular and semisimple.
See Lehrer's work \cite{lehrer} for a discussion on the regular semisimple variety in algebraic groups in both characteristic zero and positive characteristic. In particular, Lehrer gives a formula \cite[Cor.~8.5]{lehrer} enumerating the regular semisimple elements in \(\GL_n \F_q\).
Fulman gave the following combinatorial characterization of the regular semisimple elements of $\GL_n \F_q$, which also explains our choice of the notation \(\CT^\Box_\mu (q)\).
We will take Fulman's characterization as the definition of regular semisimple elements in this paper.

\begin{theorem}[Fulman \cite{fulman}]
\label{fulmans_characterization}
For all \(n \in \N\) and prime powers \(q\), a matrix $g \in \GL_n \F_q$ is regular semisimple if and only if $\ulambda^g (f) \in \{\emptyset, \Box\}$ for all $f \in \CF (q)$.
\end{theorem}

\begin{corollary}
\label{rs_rcf_corollary}
Suppose \(n \in \N\), \(q\) is a prime power, \(g \in \GL_n \F_q\) is regular semisimple and \(h_1,\ldots,h_\ell \in \CF (q)\) are the distinct irreducible factors of the characteristic polynomial of \(g\).
Then \(g\) determines the isomorphism
\begin{equation}
\label{rs_rcf}
V_g \cong
{\F_q[z]} / {(h_1(z))}
\oplus
\cdots
\oplus
{\F_q[z]} / {(h_{\ell} (z))}.
\end{equation}
\end{corollary}

Recall that we define, for \(\mu \vdash n\) and \(q\) a prime power,
\begin{equation}
\CT^\Box_\mu (q)
=
\{g \in \CT_\mu (q) : g \text{ is regular semisimple}\}.
\end{equation}
The set $\{ \CT_\mu^\Box (q) : \mu \vdash n \}$ is not a partition of $\GL_n \F_q$ in general because not all matrices in $\GL_n \F_q$ are regular semisimple.
However each $\CT^\Box_\mu (q)$ is still a union of conjugacy classes, and the set $\{\CT^\Box_\mu (q) : \mu \vdash n\}$ at least partitions the set of regular semisimple elements in $\GL_n \F_q$.

\begin{example}
Using the notation from Ex.~\ref{conj_class_example} above,
\(\CT^\Box_{(1,1,1)} (2) \) is empty,
\(\CT^\Box_{(2,1)} (2) = C_0\), and
\(\CT^\Box_{(3)} (2) = E \cup \tilde{E}\).
\end{example}

Recall that Corollary~\ref{dense} states that, for large \(q\), the set \(\CT^\Box_\mu (q)\) comprises approximately \(1 / z_\mu\) of \(\GL_n \F_q\).
We can also derive an explicit formula for \(\# \CT_\mu^\Box (q)\) by combining Theorems~\ref{cc_sizes} and \ref{fulmans_characterization}.
As mentioned by Green in \cite{green}, we have
\begin{equation}
\label{size_of_Fm}
\# \CF_m (q) = \frac{1}{m} \sum_{s | m} \moebius(m/s)(q^s-1)
\end{equation}
for all \(m \geq 1\) and prime powers \(q\),
a result originally due to Gauss in the case that \(q\) is prime \cite{gauss}.

\begin{corollary}
\label{rss_cc_and_cycle_type_sizes}
Suppose \(n \in \N\), \(\mu \vdash n\), and \(q\) is a prime power.
Then \(\CT_\mu^\Box (q)\) is a union of conjugacy classes, each with cardinality
\[
\frac{\gamma_n(q)}{\prod_{i=1}^{\ell(\mu)} (q^{\mu_i}-1)}.
\]
Therefore,
\[
\# \CT_\mu^\Box (q)
=
\frac{\gamma_n(q)}{\prod_{i=1}^{\ell(\mu)} (q^{\mu_i}-1)}
\cdot
\prod_{i \geq 1}
\binom{\# \CF_i (q)}{m_i(\mu)}.
\]
\end{corollary}

Unfortunately, we do not have an explicit formula for \(\# \CT_\mu (q)\) in general.
In fact, in Theorem~\ref{nu_n_part} and Corollary~\ref{n_minus_1_corollary}, we have technically only dealt with regular semisimple elements since \(\CT_{(n)} (q) = \CT_{(n)}^\Box (q)\) and \(\CT_{(n-1,1)} (q) = \CT_{(n-1,1)}^\Box (q)\).
As mentioned in the introduction, one can at least obtain a generating function for the sizes of the sets \(\CT_\mu (q)\) using Stong's generalization \cite{stong} of Kung's cycle index \cite{kung}.
Further discussion by Fulman appears in Section 5 of \cite{fulman}.

 %

In addition to Fulman's theorem, we will make use of the following characterization of regular semisimple elements, which appears as the final Corollary in Section 3 of Brickman-Fillmore \cite{lattice}.
Recall that a matrix \(g \in \GL_n \F_q\) is said to \myemph{stabilize} a subspace \(U \subset V\) if \(g(u) \in U\) for all \(u \in U\).
Recall also that the \myemph{lattice of stable subspaces} of a matrix $g \in \GL_n \F_q$ is the set of subspaces $U \subset V$ that \(g\) stabilizes, ordered by inclusion.

\begin{theorem}[Brickman-Fillmore \cite{lattice}] \label{helpful}
For all \(n \in \N\) and prime powers \(q\), a matrix \(g \in \GL_n \F_q\) is regular semisimple if and only if the lattice of stable subspaces of \(g\) is a Boolean lattice.
\end{theorem}

Next, we discuss another special class of matrices in $\GL_n \F_q$, the \myemph{regular elliptic} elements.
Recall from the introduction that we have defined a matrix \(g \in \GL_n \F_q\) to be regular elliptic if and only if its characteristic polynomial is irreducible.
Equivalently, the set of regular elliptic elements in \(\GL_n \F_q\) is \(\CT_{(n)} (q) = \CT_{(n)}^\Box (q)\).
This is just one of several characterizations of regular elliptic elements that we will find useful.

\begin{proposition}[{\cite[Prop.~4.4]{refl_fact}}]
\label{reg_ell_char}
For all \(n \in \N\) and prime powers \(q\), the following are equivalent for an element $g \in \GL_n \F_q$.
\begin{enumerate}[(i)]
\item The element $g$ is regular elliptic.
\item For all $x \in \GL_n \F_q$, $xgx^{-1} \in \para_\nu \implies \nu = (n)$, where \(\para_\nu\) is defined by \eqref{para_def} in Section~\ref{char_vals} below.
\item \label{no_stability} The element \(g\) stabilizes no proper nontrivial subspaces of \(V\).
\item The element \(g\) determines the isomorphism
\begin{equation}
\label{re_rcf}
V_g
\cong
\F_q[z] / (h_1(z)),
\end{equation}
where \(h_1 \in \CF_n (q)\) is the characteristic polynomial of \(g\).
\end{enumerate}
\end{proposition}

Finally, we combine the results about regular semisimple and regular elliptic elements.
The next result, which classifies the possible stable subspaces of a regular semisimple element, will be central in proving our main tool, Theorem~\ref{main_lemma_in_intro}.

\begin{corollary}[to Thm.~\ref{helpful} and Prop.~\ref{reg_ell_char}]
\label{really_helpful}
Suppose \(n \in \N\), \(q\) is a prime power, and \(g \in \GL_n \F_q\) is a regular semisimple element which determines the isomorphism
\begin{equation}
\label{really_helpful_isomorphism}
V_g
\cong
\F_q[z] / (h_1 (z)) \oplus \cdots \oplus \F_q[z] / (h_\ell (z))
\end{equation}
as in \eqref{rs_rcf}, where $h_1,\ldots,h_\ell \in \CF(q)$ are distinct and irreducible.
Suppose $g$ stabilizes a subspace $U \subset V$.
Let \(\tilde{U} \subset \bigoplus_{i=1}^{\ell} \F_q[z] / (h_i(z))\) denote the submodule corresponding to \(U\) under the isomorphism \eqref{really_helpful_isomorphism}.
Then there exists a subset \(I \subset \{1,\ldots,\ell\}\) such that \(\tilde{U} = \bigoplus_{i \in I} \F_q[z]/(h_i(z))\).
\end{corollary}

\begin{proof}
By Theorem \ref{helpful}, it suffices to show that, for each \(i \in \{1,\ldots,\ell\}\), \(g\) stabilizes no proper nontrivial subspace of $\F_q[z]/(h_i(z))$.
Consider the restriction of \(g\) to \(\F_q[z]/(h_i(z))\).
Since each $h_i$ is irreducible, the restriction of $g$ to $\F_q[z] / (h_i(z))$ is regular elliptic. By Proposition \ref{reg_ell_char}, we are done.
\end{proof}

\section{\(\GL_n \F_q\) character theory}
\label{three}

In this section, we describe how to compute the values of all the irreducible characters of $\GL_n \F_q$.
Just as with the symmetric groups, we will index the irreducible characters of \(\GL_n \F_q\) in the same way that we haved indexed its conjugacy classes.
The following is a condensed review of the topic, based on Green's work \cite{green}. The notation and language we use vary from Green's original choices. For another exposition see Macdonald \cite{macD}.

\subsection{Computing the irreducible $\GL_n \F_q$ characters} \label{char_vals}

We first introduce more notation.
For positive integers \(d\), define a function \(\alpha_d : d\Z \to \Z\) by \(\alpha_d (m) = [m/d]_{q^{d}}\).
Given a polynomial \(f \in \CF (q)\), we will consider the function \(\ell_f \alpha_{\deg f}\), which is obtained by scaling \(\alpha_{\deg f}\) by the integer \(\ell_f\).

We require a process called \myemph{parabolic induction}, which we describe now.
If $\nu = (\nu_1,\ldots,\nu_\ell) \vdash n$, let $\para_\nu$ denote the \myemph{parabolic subgroup} of $\GL_{n} \F_q$ consisting of block upper-triangular matrices with block sizes $\nu_1, \ldots, \nu_\ell$. Explicitly,
\begin{equation}
\label{para_def}
\para_\nu = \left \{
\begin{bmatrix}
A_{11} & A_{12} & \cdots & A_{1 \ell} \\
0 & A_{22} & \cdots & A_{2 \ell} \\
0 & 0 & \ddots & \vdots \\
0 & 0 & 0 & A_{\ell \ell}
\end{bmatrix} \in \GL_{n} \F_q :
A_{ii} \in \GL_{\nu_i} \F_q \text{ for all } 1 \leq i \leq \ell
\right \}.
\end{equation}
For each $i \in \{1,\ldots,\ell\}$, let $\pi_i^\nu : \para_{\nu} \to \GL_{\nu_i} \F_q$ denote projection
onto the $i^\text{th}$ diagonal block:
\begin{equation}
A = \begin{bmatrix}
A_{11} & A_{12} & \cdots & A_{1 \ell} \\
0 & A_{22} & \cdots & A_{2 \ell} \\
0 & 0 & \ddots & \vdots \\
0 & 0 & 0 & A_{\ell \ell}
\end{bmatrix}
\in \para_\nu
\implies
\pi_i^\nu (A) = A_{ii}.
\end{equation}
Given arbitrary characters $\chi_i$ of $\GL_{\nu_i} \F_q$ for each $i \in \{1,\ldots,\ell\}$,
we define their parabolic induction product $\bigodot_{i = 1}^\ell \chi_i$, which is a character of $\GL_{|\nu|} \F_q$, by
\begin{equation} \label{parabolic_induction} \left(\bigodot_{i = 1}^\ell \chi_i \right) (g) = \frac{1}{\# \para_\nu} \sum_{\substack{x \in \GL_{|\nu|} \F_q \\ xgx^{-1} \in \para_\nu}} \, \, \prod_{i = 1}^{\ell(\nu)} \chi_i \left ( \pi^\nu_i ( xgx^{-1} ) \right ). \end{equation}

We now define the irreducible characters of $\GL_n \F_q$ in four steps. First, we define the \uline{P}rimary-support characters, $P$.
Second, we define the para\uline{B}olic characters, $B$, in terms of the $P$'s. Third, we define the \uline{J}rreducible characters, $J$, in terms of the $B$'s. Finally, we define the irreducible characters \(\chi^\ulambda\) of \(\GL_n \F_q\) in terms of the $J$'s. We refer to this as the `PBJ' method. The names \myemph{Primary-support}, \myemph{paraBolic}, and \myemph{Jrreducible} were not used by Green.

For each $b \in \Z$ and $d \in \N$, we define the \uline{P}rimary-support character, $P_d^b$, of $\GL_d \F_q$ as follows. For any $\mu \in \Par \setminus \{ \emptyset \}$, let $\kappa (\mu, t) = \prod_{i=1}^{\ell(\mu) - 1} (1-t^i)$, where $\kappa(\mu, t) = 1$ if $\ell(\mu) = 1$.
Then define
\begin{equation}
\label{P_chars}
P_d^b (g) = \begin{cases} \kappa(\mu,q^{\deg h}) \sum_{i=1}^{\deg h} \te(\eps_{\deg h}^{\ell_h})^{q^i b} & \text{if } \ulambda^g = h \mapsto \mu \text{ is primary}, \\ 0 & \text{otherwise.} \end{cases}
\end{equation}
The fact that $P_d^k$ vanishes away from primary conjugacy classes explains the name \myemph{Primary-support}.
For each \(d \in \Z\), each $\nu \in \Par \setminus \{\emptyset\}$ such that \(d\) divides every part of \(\nu\), and each function $\alpha : d \Z \to \Z$, we define the para\uline{B}olic character, $B_\nu^\alpha$, of $\GL_{|\nu|} \F_q$ by
\begin{equation}
\label{B_chars}
B_\nu^\alpha = \bigodot_{i = 1}^{\ell(\nu)} P_{\nu_i}^{\alpha(\nu_i)}. \end{equation}
For each $f \in \CF (q)$ and $\lambda \in \Par$, we define the \uline{J}rreducible character, $J_f^\lambda$, of $\GL_{|\lambda| \cdot \deg f} \F_q$
by
\begin{equation}
\label{jrreducible}
J_f^\lambda
=
(-1)^{|\lambda| \cdot (\deg f - 1)} \cdot \sum_{\nu \vdash |\lambda|} \frac{1}{z_\nu} \cdot \chi^\lambda_\nu \cdot B^{\ell_f \alpha_{\deg f}}_{(\deg f) \nu}.
\end{equation}
Finally, for each index $\ulambda : \CF (q) \to \Par$ satisfying $\| \ulambda \| = n$, we define the irreducible character $\chi^\ulambda$ of $\GL_n \F_q$
by
\begin{equation} \label{grand_finale} \chi^{\ulambda} = \bigodot_{f \in \supp \ulambda} J_f^{\ulambda(f)}. \end{equation}

\begin{theorem}[Green {\cite[Theorem~14]{green}}]
\label{greens_main_thm}
For all \(n \in \N\) and prime powers \(q\), the set $\{ \chi^\ulambda : \| \ulambda \| = n\}$ is the set of distinct, irreducible, complex characters of $\GL_n \F_q$.
\end{theorem}

Green also showed that $\odot$ is commutative and associative, so it makes sense to use an arbitrary finite indexing set for the parabolic induction product. In fact, letting $J_f^\emptyset$ denote the empty function, which is the identity element with respect to $\odot$, we can define
\begin{equation}
\chi^\ulambda = \bigodot_{f \in \CF (q)} J_f^{\ulambda(f)}.
\end{equation}

Note that we have indexed the irreducible characters in the same way that we indexed the conjugacy classes.
Moreover, we also refer to an irreducible character as \myemph{primary} if its index is primary, and we use the usual \(f \mapsto \lambda\) notation for its index.
Thus, \myemph{primary characters} are those of the form \(\chi^{f \mapsto \lambda}\) for some \(f \in \CF (q)\) and \(\lambda \in \Par\).

\begin{example}
\label{runningexample}
Using the notation from Example~\ref{conj_class_example}, we record in Table~\ref{GL3F2_char_table} the character table for $\GL_3 \F_2$.
Our choice of \(\eps\) was made such that \(f_3(\eps_3) = 0\).
The rows correspond to the characters, and the columns correspond to the conjugacy classes.
In the row labels, we write, for instance, $f_1 \mapsto (2,1)$ instead of $\chi^{f_1 \mapsto (2,1)}$ for simplicity.
Let $\zeta_7 = e^{2 \pi i / 7}$.
\begin{table}[ht]
\caption{The character table of \(\GL_3 \F_2\)}
\centering
\begin{tabular}{ c || c | c | c | c | c | c }
 & $U_1$ & $U_2$ & $U_3$ & $E$ & $\tilde{E}$ & $C_0$ \\ \hline \hline
$f_1 \mapsto (1,1,1)$ & $8$ & $0$ & $0$ & $1$ & $1$ & $-1$  \\
$f_1 \mapsto (2,1)$ & $6$ & $2$ & $0$ & $-1$ & $-1$ & $0$ \\
$f_1 \mapsto (3)$ & $1$ & $1$ & $1$ & $1$ & $1$ & $1$ \\
$f_3 \mapsto (1)$ & $3$ & $-1$ & $1$ & $\zeta_7 + \zeta_7^2 + \zeta_7^4$ & $\zeta_7^3 + \zeta_7^5 + \zeta_7^6$ & $0$  \\
$\tilde{f}_3 \mapsto (1)$ & $3$ & $-1$ & $1$ & $\zeta_7^3 + \zeta_7^5 + \zeta_7^6$ & $\zeta_7 + \zeta_7^2 + \zeta_7^4$ & $0$  \\
$\ulambda_0$ & $7$ & $-1$ & $-1$ & $0$ & $0$ & $1$  \\
\end{tabular}
\label{GL3F2_char_table}
\end{table}
\end{example}

\subsection{Degrees of the irreducible $\GL_n \F_q$ characters}

Green also gives an explicit formula for the degrees of the irreducible characters $\chi^\ulambda$.
For any \(m \in \N\) and prime power \(q\), let $\psi_m (q) = (q^m - 1) \cdots (q^2-1)(q-1)$.
For any partition $\lambda$, let $b(\lambda) = \sum_{i=1}^{\ell(\lambda)} (i-1)\lambda_i$ and define
\begin{equation}
\left [ \lambda : q \right ] = q^{b(\lambda)} \frac{\prod_{1 \leq i < j \leq \ell(\lambda)} \left ( q^{(\lambda_i - \lambda_j) - (i - j)} - 1 \right )}{\prod_{i=1}^{\ell(\lambda)} \psi_{\lambda_i + \ell(\lambda) - i} (q)}.
\end{equation}

\begin{theorem}[Green {\cite[Theorem~14]{green}}]
\label{degrees}
For all \(n \in \N\), prime powers \(q\), and $\ulambda : \CF (q) \to \Par$ with $\| \ulambda \| = n$,
the degree of the irreducible character $\chi^\ulambda$ of $\GL_n \F_q$ is given by
\begin{equation}
\deg \chi^\ulambda = \psi_n (q) \prod_{f \in \CF (q)} \left [ \ulambda (f) : q^{\deg f} \right ].
\end{equation}
\end{theorem}

\begin{example}
We use Theorem~\ref{degrees} to calculate $\deg \chi^{f_1 \mapsto (2,1)}$.
By Theorem~\ref{degrees},
\[
\deg \chi^{f_1 \mapsto (2,1)}
=
\psi_3(2) \cdot [(2,1): 2]
=
\psi_3(2) \cdot \frac{6}{\psi_3(2) \cdot \psi_1(2)} = 6,
\]
which agrees with \(\chi^{f \mapsto (2,1)}\) evaluated at \(U_1\) as recorded in Table~\ref{GL3F2_char_table}.
\end{example}

\subsection{Certain character values}

Regular semisimple elements have many nice properties. The following theorem, which follows from Steinberg's work, describes one such property.
\begin{theorem}[Steinberg \cite{steinberg}]
\label{sn_stuff}
For all \(n \in \N\), prime powers \(q\), partitions \(\lambda, \mu \vdash n\), and \(g \in \CT_\mu^\Box (q)\), we have \(\chi^{z-1 \mapsto \lambda} (g) = \chi^\lambda_\mu\).
\end{theorem}

Regular elliptic elements, in particular, have nice character-theoretic properties as well.
The following result echoes Corollary~\ref{MNrule_ncycle}.
\begin{corollary}[to~
Cor.~\ref{MNrule_ncycle},
Prop.~\ref{reg_ell_char},
Thm.~\ref{greens_main_thm}, and
Thm.~\ref{degrees}]
\label{reg_ell_chi}
Suppose \(n \in \N\), \(q\) is a prime power, \(\ulambda : \CF(q) \to \Par\) with \(\|\ulambda\| = n\), and \(g \in \CT_{(n)} (q)\).
If \(\chi^\ulambda (g) \neq 0\), then there exist \(d | n\), \(f \in \CF_d (q)\), and \(r \in \{0,\ldots,n/d-1\}\) such that \(\ulambda = f \mapsto (n/d - r, 1^r)\) is primary.
Moreover,
\begin{equation}
\deg \chi^{f \mapsto (n/d-r,1^r)} = \deg_{n,d,r} (q),
\end{equation}
as defined in Table~\ref{formula_table}.
\end{corollary}

It follows that, when using the Frobenius formula to enumerate factorizations involving regular elliptic elements, one only needs to consider characters of the form \(\chi^{f \mapsto (n/d-r,1^r)}\).
Therefore, when \(n\) is understood from context, and for any \(d | n\), \(f \in \CF_d (q)\), and \(r \in \{0,\ldots,n/d-1\}\), we define
\begin{equation}
\chi^{f,r} = \chi^{f \mapsto (n/d-r,1^r)}.
\end{equation}
We can now write down a further simplified version of~\eqref{ff_in_situ}.
\begin{corollary}[to~Cor.~\ref{ff_in_situ_corollary}~and~Cor.~\ref{reg_ell_chi}]
\label{simpler_ff_in_situ_corollary}
For all \(n,k \in \N\), \(\mu \vdash n\), and prime powers \(q\), we have
\begin{equation}
\label{simpler_ff_in_situ}
g_{k,\mu} (q)
=
\frac{1}{\gamma_n(q)}
\sum_{d,f,r}
\deg_{n,d,r}(q)^{1-k}
\left ( \sum_{g \in \CT_{(n)} (q)} \chi^{f,r}(g) \right )^k
\left ( \sum_{h \in \CT_\mu (q)}   \chi^{f,r}(h) \right ),
\end{equation}
where the sum is over all \(d\) dividing \(n\), \(f \in \CF_d (q)\), and \(r \in \{0,\ldots,n/d-1\}\).
Moreover, the same is true when both \(g_{k,\mu} (q)\) is replaced with \(g_{k,\mu}^\Box (q)\) and \(\CT_\mu (q)\) is replaced with \(\CT_\mu^\Box (q)\).
\end{corollary}

\section{Proofs of main results}
\label{main_results_section}

\subsection{Main tool} \label{reg_ss}

We begin with our main tool, Theorem~\ref{main_lemma_in_intro}, a result that allows us to evaluate primary characters on regular semisimple elements more easily.
We present some lemmas before giving the proof.
In the preliminary lemmas and in Theorem~\ref{main_lemma_in_intro}, the hypotheses include ``\(\mu \vdash n\) and \(g \in \CT_\mu^\Box\).''
In all the proofs, we will assume \(g\) is in rational canonical form and denote the distinct irreducible factors of the characteristic polynomial of \(g\) by \(h_1,\ldots,h_{\ell(\mu)}\) with \(\deg h_i = \mu_i\) for each \(i \in \{1,\ldots,\ell(\mu)\}\).
Note that this implies each \(\pi_i^\mu (g) \in \GL_{\mu_i} (q)\) is in the primary conjugacy class indexed by \(h_i \mapsto \Box\).

\begin{lemma}
\label{d_divides_mu}
For all \(n \in \N\), \(d | n\), \(\nu \vdash n / d\), \(\mu \vdash n\), prime powers \(q\), \(g \in \CT_\mu^\Box (q)\), and \(\alpha : d \Z \to \Z\), we have \(B_{d\nu}^\alpha(g) = 0\) unless \(d \nu = \mu\).
\end{lemma}
\begin{proof}
By definition,
\begin{equation}
\label{eqn_for_d_divides_mu}
B_{d \nu}^\alpha (g)
=
\frac{1}{\#\para_{d \nu}}
\sum_{\substack{x \in \GL_n \F_q \\ xgx^{-1} \in \para_{d \nu}}}
\prod_{i = 1}^{\ell(\nu)}
P_{d \nu_i}^{\alpha (d \nu_i)} \left (\pi^{d \nu}_i (xgx^{-1}) \right ).
\end{equation}
Consider a single summand in \eqref{eqn_for_d_divides_mu}, which is a product of Primary-support character values.
By definition, the Primary-support characters vanish away from primary conjugacy classes.
Thus, the product of them vanishes if any diagonal block of \(xgx^{-1}\) is not primary.
So assume the product of the $P$ characters in \eqref{eqn_for_d_divides_mu} does not vanish, and hence each block \(\pi_i^{d \nu} (xgx^{-1})\) is primary.

For each \(i \in \{1,\ldots,\ell(\nu)\}\), let \(\tilde{h}_i\) be the characteristic polynomial of \(\pi^{d \nu}_i (xgx^{-1})\).
Since \(xgx^{-1} \in \para_{d \nu}\) has a block-upper-triangular structure, its characteristic polynomial equals \(\prod_{i=1}^{\ell(\nu)} \tilde{h}_i\).
The fact that each \(\pi^{d \nu}_i ( xgx^{-1})\) is primary implies that each \(\tilde{h}_i\) is a power \(\rho_i^{a_i}\) of an irreducible polynomial \(\rho_i \in \CF (q)\).
On the other hand, \(g\) is regular semisimple, meaning its characteristic polynomial has no repeated factors.
Thus, each \(a_i = 1\) and there exists a permutation \(\sigma \in \FS_{\ell(\mu)}\) such that \(\rho_i = \tilde{h}_i = h_{\sigma(i)}\) for all \(i \in \{1,\ldots,\ell(\mu)\}\).
Computing degrees show that \(d \nu_i = \deg \tilde{h}_i = \deg h_{\sigma(i)} = \mu_{\sigma(i)}\) for all \(i \in \{1,\ldots,\ell(\mu)\}\).
This implies \(d \nu = \mu\).
\end{proof}

We require some more terminology before moving forward.
Given \(\mu \vdash n\), refer to a flag $S_\bullet$ of nested subspaces
$$ S_1 \subset S_2 \subset \cdots \subset S_{\ell(\mu)}$$
of $V$
as a \myemph{$\mu$-flag} if
$$\dim S_j = \sum_{i=1}^j \mu_i$$ for all $j \in \{1,\ldots,\ell(\mu)\}$.
Refer to an ordered basis $(e_1,\ldots,e_{n})$ of $V$ as a \myemph{basis for $S_\bullet$} if
$$ \left (e_1,\ldots,e_{\sum_{i=1}^j \mu_i} \right ) $$
is a basis for $S_j$ for each $j \in \{1,\ldots,\ell(\mu)\}$.
Conversely, each ordered basis \((e_1,\ldots,e_n)\) for \(V\) \myemph{determines a \(\mu\)-flag} by taking the \(j^\text{th}\) subspace in the flag to be the span of \(\left ( e_1,\ldots,e_{\sum_{i=1}^j \mu_i} \right )\) for each \(j \in \{1,\ldots,\ell(\mu)\}\).
Given a $\mu$-flag $S_\bullet$, say that a matrix in $\GL_{n} \F_q$ \myemph{stabilizes $S_\bullet$} if it stabilizes $S_j$ for each $j \in \{1,\ldots,\ell(\mu)\}$.

\begin{lemma}
\label{parabolic_counting_lemma}
For all \(n \in \N\), $\mu \vdash n$, prime powers \(q\), and $g \in \CT_\mu^\Box (q)$, we have
\begin{equation} \label{parabolic_counting_eqn} \# \left \{ x \in \GL_n \F_q : xgx^{-1} \in \para_\mu \right \} = \# \para_\mu \cdot \prod_{i \geq 1} m_i(\mu)! .\end{equation}
\end{lemma}
\begin{proof}
Viewing \(g\) abstractly as a linear transformation on \(V\), the left side of \eqref{parabolic_counting_eqn} is the number of ordered bases of $V$ with respect to which the matrix representing $g$ is an element of $\para_\mu$.
For any fixed basis $\CB = (v_1,\ldots,v_{n})$ for $V$, being an element of $\para_\mu$ is equivalent to stabilizing the $\mu$-flag determined by $\CB$.
Therefore, the left side of \eqref{parabolic_counting_eqn} is the product of
\begin{enumerate}[(1)]
\item the number of $\mu$-flags that $g$ stabilizes and
\item the number of ordered bases for a given $\mu$-flag.
\end{enumerate}
The second part is $\# \para_\mu$.

The first part is slightly less immediate.
Consider a $\mu$-flag
$$ S_\bullet = S_1 \subset S_2 \subset \cdots \subset S_{\ell(\mu)}$$
that $g$ stabilizes.
We must count how many choices there are for \(S_\bullet\).
For each \(i \in \{1,\ldots,\ell(\mu)\}\), let \(Q_i\) denote the quotient \(S_i / S_{i-1}\), with the convention that \(S_0\) is zero-dimensional.
Note that \(\dim Q_i = \mu_i\).
Under the isomorphism \eqref{rs_rcf}, let \(E_i \subset V\) denote the subspace corresponding to \(\F_q[z] / (h_i(z))\) for each \(i \in \{1,\ldots,\ell(\mu)\}\).
By Corollary \ref{really_helpful}, the only subspaces of $V$ that $g$ stabilizes are direct sums of the $E_i$'s.
Therefore, each $S_j$ is a direct sum of the $E_i$'s.
This implies each $Q_i$ is also a direct sum of the $E_i$'s.
Working backwards from \(Q_{\ell(\mu)}\) to \(Q_1\), we see that
\(
Q_{\ell(\mu)-m_1(\mu)+1},
\ldots,
Q_{\ell(\mu)}
\) are all 1-dimensional,
and thus must be a permutation of
\(
E_{\ell(\mu)-m_1(\mu)+1},
\ldots,
E_{\ell(\mu)}
\).
There are therefore \(m_1(\mu)!\) ways to choose the final \(m_1(\mu)\) quotients.
Likewise,
\[
Q_{\ell(\mu)-m_1(\mu)-m_2(\mu)+1},
\ldots,
Q_{\ell(\mu)-m_1(\mu)}
\]
are all 2-dimensional, and thus must be a permutation of
\[
E_{\ell(\mu)-m_1(\mu)-m_2(\mu)+1},
\ldots,
E_{\ell(\mu)-m_1(\mu)},
\]
as there are no 1-dimensional \(E_i\)'s remaining, giving \(m_2(\mu)!\) choices for those quotients.
Continuing, we see that there are \(\prod_{i \geq 1} m_i(\mu)!\) choices for \(Q_1,\ldots,Q_{\ell(\mu)}\).
Observing that the quotients determine \(S_\bullet\) uniquely, the result follows.
\end{proof}

We are now ready to prove Theorem~\ref{main_lemma_in_intro}.
Recall that given \(d | n\), \(f \in \CF_d (q)\), and \(\lambda \vdash n/d\), it provides a simple formula for \(\chi^{f \mapsto \lambda} (g)\).
Specifically, Theorem~\ref{main_lemma_in_intro} states that,
if some part of \(\mu\) is not divisible by \(d\), then \(\chi^{f \mapsto \lambda} (g) = 0\), and otherwise, there exists \(\tilde{\mu} \vdash n/d\) such that \(\mu = d \tilde{\mu}\), and we have
\begin{equation}
\chi^{f \mapsto \lambda} (g)
= (-1)^{\tfrac{n}{d}(d-1)}
\chi^\lambda_{\tilde{\mu}}
\prod_{i = 1}^{\ell(\mu)}
\frac{1}{\tilde{\mu}_i}
\sum_{\substack{ \beta_i \in \F_{q^{\mu_i}} \\ h_i(\beta_i) = 0}} \te (\beta_i)^{\ell_f [\tilde{\mu}_i]_{q^d}}.
\end{equation}

\begin{proof}[Proof of Theorem \ref{main_lemma_in_intro}]
By definitions \eqref{jrreducible} and \eqref{grand_finale}, we have
\begin{equation}
\label{in_the_beginning}
\chi^{f \mapsto \lambda} (g)
= J_f^\lambda (g) = (-1)^{\tfrac{n}{d} (d-1)} \cdot \sum_{\nu \vdash \tfrac{n}{d}} \frac{\chi^\lambda_\nu}{z_\nu} B_{d \nu}^{\ell_f \alpha_{d}} (g)
\end{equation}
By Lemma \ref{d_divides_mu}, $B_{d \nu}^{\ell_f \alpha_d} (g) = 0$ unless $d \nu = \mu$.
If some part of $\mu$ is not divisible by $d$, then $\chi^{f \mapsto \lambda}(g) = 0$, proving the first statement in the lemma.
Otherwise, there exists a unique partition \(\tilde{\mu} \vdash n/d\) such that \(\mu = d \tilde{\mu}\),
and only the summand corresponding to \(\tilde{\mu}\) in \eqref{in_the_beginning} does not vanish.
Therefore, \eqref{in_the_beginning} reduces to
\begin{equation}
\label{super_concise}
\chi^{f \mapsto \lambda} (g) = (-1)^{\frac{n}{d}(d-1)} \frac{\chi^\lambda_{\tilde{\mu}}}{z_{\tilde{\mu}}} B_{\mu}^{\ell_f \alpha_d} (g).
\end{equation}
By definitions \eqref{parabolic_induction} and \eqref{B_chars}, we can rewrite \eqref{super_concise} as
\begin{equation}
\label{expandedB_eqn}
\chi^{f \mapsto \lambda} (g)
= (-1)^{\tfrac{n}{d} (d-1)} \cdot \frac{\chi^\lambda_{\tilde{\mu}}}{z_{\tilde{\mu}}} \cdot \frac{1}{\# \para_{\mu}} \sum_{\substack{x \in \GL_n \F_q \\ xgx^{-1} \in \para_{\mu}}} \prod_{i = 1}^{\ell(\mu)} P_{\mu_i}^{\ell_f \alpha_d (\mu_i)} (\pi^\mu_i (xgx^{-1})).
\end{equation}

Consider the summation in \eqref{expandedB_eqn}.
It can be rewritten as
\begin{equation}
\label{expandedB_eqn_summation}
\sum_{\substack{x \in \GL_n \F_q \\ xgx^{-1} \in \para_\mu}}
\prod_{s \geq 1}
\prod_{\{j \in \N \, : \, \mu_j = s\}}
P_s^{\ell_f \alpha_d(s)}
(\pi_j^\mu (xgx^{-1})).
\end{equation}
For each $x$ such that $xgx^{-1} \in \para_\mu$, consider the corresponding summand in \eqref{expandedB_eqn_summation}.
We repeat a similar argument to the one presented toward the end of the proof of Lemma~\ref{parabolic_counting_lemma}.
Observe that the characteristic polynomials of
\(
\{\pi_{j}^\mu (xgx^{-1})
: j \in \N, \, \mu_j = 1\}\)
have degree 1 and hence are a permutation of
\(
\{h_j : j \in \N, \, \mu_j = 1\}.
\)
Likewise, the characteristic polynomials of
\(
\{\pi_{j}^\mu (xgx^{-1}) : j \in \N, \, \mu_j = 2\}\)
have degree 2 and hence are a permutation of
\(
\{h_j : j \in \N, \, \mu_j = 2\},
\)
as there are no degree-1 factors remaining.
Continuing, we see that, for each \(s \in \N\), the degree-\(s\) characteristic polynomials of the diagonal blocks of \(xgx^{-1}\) are a permutation of the degree-\(s\) irreducible factors of the characteristic polynomial of \(g\).
Moreover, the value of each \(P_s^{\ell_f \alpha_d(s)}\) in \eqref{expandedB_eqn_summation} depends only on the characteristic polynomial of the argument.
This implies that the product of the Primary-support characters in \eqref{expandedB_eqn_summation} is constant over the sum.
Therefore,
\begin{equation}
\label{pre_near_the_end}
\chi^{f \mapsto \lambda} (g) = (-1)^{\tfrac{n}{d} (d-1)} \cdot \frac{\chi^\lambda_{\tilde{\mu}}}{z_{\tilde{\mu}}} \cdot \frac{\# \{x \in \GL_n \F_q : xyx^{-1} \in \para_\mu \} }{\# \para_{\mu}} \cdot \prod_{i = 1}^{\ell(\mu)} P_{\mu_i}^{\ell_f \alpha_d (\mu_i)} (\pi^\mu_i (g)).
\end{equation}
By Lemma \ref{parabolic_counting_lemma}, \eqref{pre_near_the_end} reduces to
\begin{equation}
\label{near_the_end}
\chi^{f \mapsto \lambda} (g) = (-1)^{\tfrac{n}{d}(d-1)} \cdot
\chi^\lambda_{\tilde{\mu}}
\cdot
\prod_{i = 1}^{\ell(\mu)}
\frac{1}{\tilde{\mu}_i}
P_{\mu_i}^{\ell_f \alpha_d (\mu_i)} (\pi^\mu_i(g)).
\end{equation}

Consider the Primary-support character evaluations in \eqref{near_the_end}.
By \eqref{who_are_my_roots} and definition \eqref{P_chars},
\begin{align}
P_{\mu_i}^{\ell_f \alpha_f (\mu_i)} (\pi^\mu_i (g))
& = \kappa(\Box,q^{\deg h_i}) \sum_{j=1}^{\deg h_i} \te(\eps_{\deg h_i}^{\ell_{h_i}})^{q^j \ell_f \alpha_d (\mu_i)} \\
& = \sum_{j = 1}^{\mu_i} \te(\eps_{\mu_i}^{\ell_{h_i}})^{q^j \ell_f  [\tilde{\mu}_i]_{q^d}}
= \sum_{\substack{\beta_i \in \F_{q^{\mu_i}} \\ h_i( \beta_i) = 0}} \te ( \beta_i )^{\ell_f [\tilde{\mu}_i]_{q^d}}. \label{very_last_step}
\end{align}
Substituting \eqref{very_last_step} into \eqref{near_the_end}
gives the result.
\end{proof}

\begin{example}
Returning to Example~\ref{runningexample}, we compute $\chi^{f_3 \mapsto (1)} (c)$ for $c \in E$.
In the notation of Theorem \ref{main_lemma_in_intro}, we have $d = 3, f = f_3, \lambda = (1), \mu = (3), \tilde{\mu} = (1), h_1 = f_3$, and $\ell_f = 1$.
The roots of $h_1$ are $\eps_3, \eps_3^2$, and $\eps_3^4$.
By Theorem \ref{main_lemma_in_intro},
\begin{align*}
\chi^{f_3 \mapsto (1)} (c)
& =
(-1)^{\frac{3}{3}(3-1)}
\cdot
\chi^{(1)}_{(1)}
\cdot
\sum_{h_1(\beta) = 0} \te(\beta) \\
& = \te(\eps_3) + \te(\eps_3^2) + \te(\eps_3^4)
= \zeta_7 + \zeta_7^2 + \zeta_7^4.
\end{align*}
We see this exact value in Table~\ref{GL3F2_char_table} above.
\end{example}

\subsection{Proof of first main result}

We can now apply Theorem~\ref{main_lemma_in_intro} to prove our main results.
We begin with Theorem~\ref{re_main}, which addresses the family of cases \(\mu = (\mu_1, \cdots, \mu_\ell) \vdash n > 2\), where \(\ell > 1\) and \(\mu_{\ell-1} > \mu_\ell = 1\).
Note that if all the parts of \(\mu\) are distinct, i.e. \(\mu_1 > \cdots > \mu_\ell\), then \(\CT_\mu^\Box (q) = \CT_\mu (q)\), and consequently \(g_{k,\mu}(q) = g_{k,\mu}^\Box (q)\).
However, even in the event that some of the parts of \(\mu\) agree, we are still able to find an explicit formula for \(g_{k,\mu}^\Box (q)\), although \(\CT_\mu^\Box (q)\) might be empty if \(q\) is small.
Recall that Theorem~\ref{re_main} states that, under the above assumptions on \(\mu\) and \(n\),
\begin{equation}
g_{k,\mu}^\Box(q)
=
\frac{\# \CT_{(n)} (q)^k
\cdot
\# \CT_\mu^\Box (q)}{\# \GL_n \F_q}
\cdot
\sum_{r=0}^{n-1}
\frac{(-1)^{rk} \chi^{(n-r,1^r)}_\mu}{\left ( q^{\binom{r+1}{2}} \cdot {\genfrac[]{0pt}{1}{n - 1}{r}}_{q} \right )^{k-1}}
\end{equation}
for all \(k \in \N\) and prime powers \(q\).

In the rest of this section and later, we will require some additional notation.
We will consider the logical propositions ``\(q-1 | \ell_f\)'' for various \(f \in \CF_1 (q)\).
Even though \(\ell_f\) denotes an arbitrary choice, these propositions are well-defined for the following reason.
Suppose for the chosen \(\ell_f\), we have \(q-1 | \ell_f\).
Now, suppose \(\ell_f'\) is another choice of \(\ell_f\).
Then there exist \(i,j \in \Z\) such that
\begin{equation}
\ell_f' = q^i \ell_f + j(q-1),
\end{equation}
which is divisible by \(q-1\) by the assumption that \(q-1 | \ell_f\).

\begin{proof}[Proof of Theorem~\ref{re_main}]
By Corollary~\ref{simpler_ff_in_situ_corollary}, we have
\[
g_{k,\mu}^\Box (q)
=
\frac{1}{\gamma_n(q)}
\sum_{d | n}
\sum_{r = 0}^{\tfrac{n}{d} - 1}
\deg_{n,d,r}(q)^{1-k}
\sum_{f \in \CF_d (q)}
\left (
\sum_{g \in \CT_{(n)} (q)} \chi^{f,r} (g)
\right )^k
\left (
\sum_{h \in \CT^\Box_\mu (q)} \chi^{f,r} (h)
\right ).
\]
Applying Theorem~\ref{main_lemma_in_intro}, we see that \(\chi^{f,r}\) vanishes on \(\CT_\mu^\Box (q)\) unless \(d = 1\).
Therefore,
\begin{equation}
\label{nepenthe}
g_{k,\mu}^\Box (q)
=
\frac{1}{\gamma_n(q)}
\sum_{r = 0}^{n - 1}
\deg_{n,1,r}(q)^{1-k}
\sum_{f \in \CF_1 (q)}
\left (
\sum_{g \in \CT_{(n)} (q)} \chi^{f,r} (g)
\right )^k
\left (
\sum_{h \in \CT_\mu^\Box (q)} \chi^{f,r} (h)
\right ).
\end{equation}

We proceed to show that only the \(f(z) = z-1\) term does not vanish in the sum over \(f \in \CF_1(q)\).
Consider an individual polynomial \(f \in \CF_1 (q)\).
Recall from Section~\ref{indexing} that the conjugacy classes in \(\CT_\mu^\Box (q)\) have equals sizes and each conjugacy class is uniquely determined by a set \(\{h_1,\ldots,h_\ell\}\) of distinct polynomials such that \(h_i \in \CF_{\mu_i} (q)\) for each \(i \in \{1,\ldots,\ell\}\).
Thus,
by Theorem~\ref{main_lemma_in_intro},
\(\sum_{h \in \CT_\mu^\Box (q)} \chi^{f,r} (h)\) is a multiple of
\begin{equation}
\label{ada_cafe}
\sum_{\substack{\{h_1,\ldots,h_\ell\} \\ h_i \in \CF_{\mu_i} (q)}}
\prod_{i=1}^\ell
\sum_{\substack{\beta_i \in \F_{q^{\mu_i}} \\ h_i(\beta_i) = 0}} \te(\beta_i)^{\ell_f [\mu_i]_q}.
\end{equation}
Since \(\mu_\ell\) is the only part of \(\mu\) equal to \(1\), we have that \eqref{ada_cafe} factors as
\begin{equation}
\label{ada_cafe_2}
\left (
\sum_{\substack{\{h_1,\ldots,h_{\ell-1}\} \\ h_i \in \CF_{\mu_i} (q)}}
\prod_{i=1}^{\ell-1}
\sum_{\substack{\beta_i \in \F_{q^{\mu_i}} \\ h_i(\beta_i) = 0}} \te(\beta_i)^{\ell_f [\mu_i]_q}
\right )
\cdot
\left (
\sum_{h_\ell \in \CF_{1} (q)}
\sum_{\substack{\beta_\ell \in \F_{q} \\ h_\ell(\beta_\ell) = 0}} \te(\beta_\ell)^{\ell_f}
\right ).
\end{equation}
The latter factor in \eqref{ada_cafe_2} is
\[
\sum_{h_\ell \in \CF_1 (q)} \sum_{\substack{\beta_\ell \in \F_q \\ h_\ell(\beta_\ell) = 0}} \te(\beta_\ell)^{\ell_f}
=
\begin{cases}
0, & q-1 \nmid \ell_f, \\
q-1, & q-1 \mid \ell_f,
\end{cases}
\]
because, by Corollary~\ref{te_n_map_property}, \(\te(\beta_\ell)\) ranges over all \((q-1)^\text{th}\) roots of unity.
Since \(\deg f = d = 1\), we can take \(\ell_f \in \{1,\ldots,q-1\}\).
Thus, the only non-zero contribution to the sum over \(f \in \CF_1(q)\) in \eqref{nepenthe} comes from the term corresponding to \(\ell_f = q-1\) and hence \(f(z) = z-1\).

Eliminating the vanishing terms not corresponding to \(f(z) = z-1\) in \eqref{nepenthe} gives
\[
g_{k,\mu}^\Box (q)
=
\frac{1}{\gamma_n(q)}
\sum_{r = 0}^{n-1}
\deg_{n,1,r}(q)^{1-k}
\left (
\sum_{g \in \CT_{(n)} (q)} \chi^{z-1,r} (g)
\right )^k
\left (
\sum_{h \in \CT_\mu^\Box (q)} \chi^{z-1,r} (h)
\right ).
\]
By
Corollary~\ref{MNrule_ncycle}
and
Theorem~\ref{sn_stuff},
\begin{align*}
g \in \CT_{(n)} (q) & \implies \chi^{z-1,r}(g) = (-1)^r, \quad \quad \text{and}\\
h \in \CT_\mu^\Box (q)   & \implies \chi^{z-1,r}(h) = \chi^{(n-r,1^r)}_\mu.
\end{align*}
Since both character values only depend on \(r\), we have
\[
g_{k,\mu}^\Box (q)
=
\frac{1}{\gamma_n(q)}
\sum_{r = 0}^{n-1}
\deg_{n,1,r}(q)^{1-k}
\left (
\# \CT_{(n)} (q) (-1)^r
\right )^k
\left (
\# \CT_\mu^\Box (q) \chi^{(n-r,1^r)}_\mu
\right ),
\]
which simplifies to the result, using the notation from Table~\ref{formula_table}.
\end{proof}

Next is the case of \(\mu = (n-1,1)\), which is addressed by Corollary~\ref{n_minus_1_corollary}.
In this case, \(g_{k,(n-1,1)} (q)\) equals the number of \(k\)-tuples of regular elliptic elements whose product has exactly one eigenvalue in \(\F_q\) and acts as a regular elliptic element on an \((n-1)\)-dimensional subspace of \(V\).
Recall that Corollary~\ref{n_minus_1_corollary} states that
\begin{equation}
g_{k,(n-1,1)} (q)
=
\frac{\# \CT_{(n)} (q)^k
\cdot
\# \CT_{(n-1,1)} (q)}{\# \GL_n \F_q}
\cdot
\left (
1 + \frac{(-1)^{nk-n-k}}{q^{\binom{n}{2}(k-1)}}
\right )
\end{equation}
for all \(n > 2\), \(k \in \N\), and prime powers \(q\).

\begin{proof}[Proof of Corollary~\ref{n_minus_1_corollary}]
Apply Corollary~\ref{MNrule_ncycle} to Theorem~\ref{re_main},
observing that \(\chi^{(n-r,1^r)}_{(n-1,1)}\) is only nonzero when \(r \in \{0,n-1\}\).
\end{proof}

\subsection{Proof of second main result}
\label{nu_equals_n_section}

Our second main result, Theorem~\ref{nu_n_part}, addresses the case \(\mu = (n)\).
In this case, the quantity \(g_{k,(n)} (q)\) equals the number of \(k\)-tuples of regular elliptic elements whose product is also regular elliptic.
Recall the notation and definitions from Table~\ref{formula_table}, and recall that Theorem~\ref{nu_n_part} states that
\begin{equation}
g_{k,(n)}(q) = P_{n,k+1} (q)
\sum_{d | n}
(-1)^{n(k+1)/d}
d^{k}
D_{n,k+1,d} (q)
\sum_{c | d}
\moebius(d/c)
C_{n,k+1,c} (q)
\end{equation}
for all \(n, k \in \N\) and prime powers \(q\).

\begin{proof}[Proof of Theorem~\ref{nu_n_part}.]
By Corollary~\ref{simpler_ff_in_situ_corollary} and the fact that \(\mu = (n)\), we have
\begin{equation}
\label{first}
g_{k,\mu} (q)
=
\frac{1}{\gamma_n(q)}
\sum_{d | n}
\sum_{r=0}^{\tfrac{n}{d}-1}
\deg_{n,d,r}(q)^{1-k}
\sum_{f \in \CF_d (q)}
\left ( \sum_{g \in \CT_{(n)} (q)} \chi^{f,r}(g) \right )^{k+1}.
\end{equation}
By Theorem~\ref{cc_sizes}, the size of each conjugacy class comprising \(\CT_{(n)} (q)\) is \(\gamma_n(q) / (q^n-1)\).
Since characters are constant on conjugacy classes, we have
\begin{equation}
\sum_{g \in \CT_{(n)} (q)}
\chi^{f,r} (g)
=
\frac{\gamma_n(q)}{q^n-1} \sum_{p \in \CF_n (q)} \chi^{f,r}(g_p),
\label{cc_char_sum_reduc}
\end{equation}
where \(g_p\) denotes an arbitrary regular elliptic element with characteristic polynomial \(p\).
Substituting \eqref{cc_char_sum_reduc} into \eqref{first}, we have
\begin{equation}
\label{almost_there}
g_{k,\mu} (q)
=
\frac{1}{\gamma_n(q) }
\left ( \frac{\gamma_n(q)}{q^n-1} \right )^{k+1}
\sum_{d | n}
\sum_{r = 0}^{\tfrac{n}{d} - 1} \deg_{n,d,r}(q)^{1-k} \sum_{f \in \CF_d (q)} \left ( \sum_{p \in \CF_n (q)} \chi^{f,r}(g_p) \right )^{k+1}.
\end{equation}
Observe that \(\CT_{(n)} (q) = \CT_{(n)}^\Box (q)\), so we can evaluate primary characters on \(\CT_{(n)} (q)\) using Theorem~\ref{main_lemma_in_intro}.
Applying Theorem~\ref{main_lemma_in_intro} and Corollary~\ref{MNrule_ncycle} to \eqref{almost_there} gives
\begin{align*}
g_{k,\mu} (q)
& =
\frac{1}{\gamma_n(q) }
\left ( \frac{(-1)^n \gamma_n(q)}{n(q^n-1)} \right )^{k+1}
\sum_{d | n} ((-1)^{n/d} d)^{k+1}
\sum_{r = 0}^{\tfrac{n}{d} - 1} (-1)^{r(k+1)} \deg_{n,d,r}(q)^{1-k} \\
& \quad \quad \quad \times
\sum_{f \in \CF_d (q)} \left ( \sum_{p \in \CF_n (q)} \sum_{\substack{\alpha \in \F_{q^n}^\times \\ p(\alpha) = 0}}
\te(\alpha)^{\ell_f [n/d]_{q^d}} \right )^{k+1}.
\end{align*}
Using the notation established in Table~\ref{formula_table}, this is equivalent to
\begin{equation}
\frac{g_{k,\mu} (q)}{P_{n,k+1}(q)}
=
\sum_{d | n} ((-1)^{n/d} d)^{k+1}
D_{n,k+1,d}(q)
\sum_{f \in \CF_d (q)} \left ( \sum_{p \in \CF_n (q)} \sum_{\substack{\alpha \in \F_{q^n}^\times \\ p(\alpha) = 0}}
\te(\alpha)^{\ell_f [n/d]_{q^d}} \right )^{k+1}.
\end{equation}
Theorem~\ref{nu_n_part} now follows from Corollary~\ref{pre_main} below, which we phrase in terms of \(k\) rather than \(k+1\) for the sake of simplifying the expressions.
\end{proof}

Before proving Corollary~\ref{pre_main}, we prove two lemmas.
Given a logical proposition \(\CP\), let \(\delta_\CP\) equal \(1\) if \(\CP\) is true and \(0\) if \(\CP\) is false.
We will make logical propositions of the form ``\(b | \ell_f [n/d]_{q^d}\)'' or equivalently ``\(b\) divides \(\ell_f [n/d]_{q^d}\)'', where \(d | n\), \(f \in \CF_d (q)\), and \(b\) is a number that divides \(q^n-1\).
These are not of the same form as ``\(q-1 | \ell_f\),'' which we considered earlier.
However, they are still well-defined for the following reason.
First, \(b\) dividing an element of \(\Z / (q^n-1)\) is well-defined simply because \(b\) itself divides \(q^n-1\).
Second, suppose for the chosen \(\ell_f\), we have \(b | \ell_f [n/d]_{q^d}\), where \(d | n\) and \(b | q^n-1\).
Now, suppose \(\ell_f'\) is another choice of \(\ell_f\).
Then there exist \(i,j \in \Z\) such that
\begin{equation}
\label{well_defined_div}
\ell_f' = q^i \ell_f + j(q^d-1).
\end{equation}
Observe that \((q^d-1)[n/d]_{q^d} = q^n-1\).
Multiplying both sides of \eqref{well_defined_div} by \([n/d]_{q^d}\), we have
\begin{equation}
\ell_f'[n/d]_{q^d} = q^i \ell_f [n/d]_{q^d} + j (q^n-1),
\end{equation}
which is divisible by \(b\) by the assumption that \(b\) divides both \(\ell_f [n/d]_{q^d}\) and \(q^n-1\).


\begin{lemma}
\label{pre_pre_main}
For all \(n \in \N\), \(d | n\), prime powers \(q\), and \(f \in \CF_d (q)\),
\begin{equation}
\sum_{p \in \CF_n (q)}
\sum_{\substack{\alpha \in \F_{q^n}^\times \\ p(\alpha) = 0}}
\te(\alpha)^{\ell_f [n/d]_{q^d}}
=
\sum_{s | n} \moebius(n/s) (q^s-1) \delta_{q^s - 1 | \ell_f [n/d]_{q^d}}.
\end{equation}
\end{lemma}


\begin{proof}
By M\"obius inversion, it suffices to prove
\begin{equation}
\label{inverted}
\sum_{s | n}
\sum_{p \in \CF_s (q)}
\sum_{\substack{\alpha \in \F_{q^s}^\times \\ p(\alpha) = 0}}
\te (\alpha)^{\ell_f [n/d]_{q^d}}
=
(q^n - 1) \delta_{q^n - 1 | \ell_f [n/d]_{q^d}}.
\end{equation}
Corollary~\ref{theta_of_root_union} implies that, on the left side of~\eqref{inverted}, \(\te(\alpha)\) ranges precisely over all \((q^n-1)^\text{th}\) roots of unity.
The sum of the \((\ell_f [n/d]_{q^d})^\text{th}\) powers of all \((q^n-1)^\text{th}\) roots of unity is zero unless \(q^n-1\) divides \(\ell_f [n/d]_{q^d}\), in which case the sum is \(q^n-1\).
\end{proof}


\begin{lemma}
\label{post_pre_pre_main}
For all \(n \in \N, \, d | n\), prime powers \(q\), and \(b | q^n-1\),
\begin{equation}
\# \{f \in \CF_d (q) : \, \, b | \ell_f [n/d]_{q^d}\}
=
\frac{1}{d}
\sum_{c | d}
\moebius(d/c)
\frac{q^n-1}{\lcm \big (
[n/c]_{q^c}, b
\big )}
\end{equation}
\end{lemma}


\begin{proof}
By M\"obius inversion, it suffices to prove
\begin{equation}
\label{inverted_2}
\sum_{c | d} c \cdot \# \{f \in \CF_c (q) : \, \, b | \ell_f [n/c]_{q^c} \}
=
\frac{q^n-1}{\lcm \big (
[n/d]_{q^d}, b
\big )}.
\end{equation}
View each value \(\ell_f [n/c]_{q^c}\) as an element of \(\Z / (q^n-1)\), as in the context of Corollary~\ref{te_n_map_property}.
Modulo \(q^n-1\), there are exactly \(c\) distinct choices for each \(\ell_f [n/c]_{q^c}\).
Namely, given a choice for \(\ell_f\),
\[
\ell_f , q\ell_f , q^2\ell_f , \ldots, q^{c-1}\ell_f \quad \mod q^n-1
\]
are also valid choices for \(\ell_f\), and so
\[
\ell_f [n/c]_{q^c}, q\ell_f[n/c]_{q^c} , q^2\ell_f[n/c]_{q^c} , \ldots, q^{c-1}\ell_f[n/c]_{q^c} \quad \mod q^n-1
\]
are all the possible choices for \(\ell_f [n/c]_{q^c}\) in \(\Z / (q^n-1)\).
Recalling definition~\eqref{te_n_def}, we see that those values are precisely the images under \(\te_n\) of the roots of \(f\).
Thus, another way to interpret the sum on the left side of \eqref{inverted_2} is
\[
\sum_{c | d} \sum_{f \in \CF_c (q)} \sum_{\substack{\alpha \in \F_{q^c} \\ f(\alpha) = 0}} \delta_{b | \te_n(\alpha)}.
\]
Therefore, by Corollary~\ref{te_n_map_property}, the left side of \eqref{inverted_2} counts the elements of \(\Z / (q^n-1)\) that are divisible by both \([n/d]_{q^d}\) and \(b\),
which is exactly the right side of \eqref{inverted_2}.
\end{proof}


\begin{corollary}[to Lem.~\ref{pre_pre_main} and Lem.~\ref{post_pre_pre_main}]
\label{pre_main}
For all \(n,k \in \N, \, d | n\), and prime powers \(q\),
\[
\sum_{f \in \CF_d (q)}
\left (
\sum_{p \in \CF_n (q)}
\sum_{\substack{\alpha \in \F_{q^n}^\times \\ p(\alpha) = 0}}
\te(\alpha)^{\ell_f [n/d]_{q^d}}
\right )^k
=
\frac{1}{d} \sum_{c | d} \moebius(d/c) C_{n,k,c}(q).
\]
\end{corollary}


\begin{proof}
Using Lemma~\ref{pre_pre_main} and expanding the \(k^\text{th}\) power in the statement, it remains to show
\[
\sum_{s_1,\ldots,s_k | n}
\# \{f \in \CF_d (q) : \lcm(q^{s_1}-1,\ldots,q^{s_k}-1) | \ell_f [n/d]_{q^d}\}
\cdot
\prod_{i=1}^k \moebius(n/s_i)(q^{s_i}-1)
\]
equals
\[
\frac{1}{d} \sum_{c | d} \moebius(d/c) C_{n,k,c} (q).
\]
By the definition of \(C_{n,k,c} (q)\),
this is equivalent to showing that
\[
\# \{f \in \CF_d (q) : \lcm(q^{s_1}-1,\ldots,q^{s_k}-1) | \ell_f [n/d]_{q^d} \}
\]
equals
\[
\frac{1}{d}
\sum_{c | d}
\moebius(d/c)
\frac{q^n-1}{\lcm \big ( [n/c]_{q^c}, q^{s_1}-1, \ldots, q^{s_k}-1 \big )}.
\]
This follows from Lemma~\ref{post_pre_pre_main}.
\end{proof}

\section{Probabilistic result}
\label{probabilistic_section}

\subsection{Prerequisite lemmas}

We require some lemmas before proving Theorem~\ref{intro_version_p}, and we introduce more notation to do so.
For a complex number \(\alpha\), let \(\bar{\alpha}\) denote the complex conjugate of \(\alpha\), and let \(\| \alpha \| = \sqrt{\alpha \bar{\alpha}} \in [0,\infty)\) denote the usual \myemph{norm} of \(\alpha\).
Define the function \(D : \N \to \N \cup \{0\}\) by
\begin{equation}
D(n) =
\begin{cases}
0 & \text{ if } n = 1, \\
\max \{ s \in \N : s | n \text{ and } s < n\} & \text{ if } n > 1.
\end{cases}
\end{equation}
In other words, \(D(n)\) is the largest proper divisor of \(n\), unless \(n = 1\), and \(D(1) = 0\).
Note that \(D(n) \leq n/2\) for all \(n \in \N\).
In this section, we also make use of \myemph{big \(O\) notation}.
Recall that if \(S\) is an infinite subset of \(\N\) and \(f,g : S \to [0,\infty)\), then we write \(f = O(g)\) to denote that there exist \(m \in [0,\infty)\) and \(s_0 \in S\) such that \(f(s) \leq m \cdot g(s)\) for all \(s > s_0\).
In other words, \(f\) is big \(O\) of \(g\) if some constant multiple of \(g(s)\) is an upper bound on \(f(s)\) for all sufficiently large \(s \in S\).

\begin{lemma}
\label{degree_limit_lemma}
For all \(n \in \N, \, d | n\), and \(r \in \{0,\ldots,\tfrac{n}{d}-1\}\), we have
\begin{equation}
\frac{1}{\deg_{n,d,r} (q)}
=
O \left (
q^{-e(n,d,r)}
\right ),
\end{equation}
where we define
\begin{equation}
\label{definition_of_e}
e(n,d,r)
=
d\binom{r+1}{2} + \binom{n+1}{2} - d\binom{\tfrac{n}{d} + 1}{2} + dr\left ( \tfrac{n}{d}-1-r \right ).
\end{equation}
Moreover, \(e(n,d,r)\) is positive unless \(d = 1\) and \(r = 0\), and \(e(n,1,0) = 0\).
\end{lemma}

\begin{proof}
The first claim follows from the discussion in Remark~\ref{more_notation_and_rationality} and then checking the degrees in \(q\) of the various polynomials which comprise \(\deg_{n,d,r} (q)\).
To prove the second claim, first observe that \(e(n,d,r)\) is a quadratic polynomial in \(r\) with critical point \(r = \tfrac{n}{d} - \tfrac{1}{2}\) and leading coefficient \(-d/2\).
This implies \(e(n,d,r)\) is increasing for \(r \in \{0,\ldots,\tfrac{n}{d}-1\}\).
The minimum value of \(e(n,d,r)\) on \(\{0,\ldots,\tfrac{n}{d}-1\}\) is therefore achieved at \(r = 0\).
Observe that \(e(n,d,0) = \tfrac{n}{2}\left ( n - \tfrac{n}{d}\right )\), which is positive unless \(d = 1\).
Therefore, \(e(n,d,r) > 0\) if \(d > 1\).
In the case \(d = 1\), we have \(e(n,1,r) = -\tfrac{1}{2}r^2 + (n-\tfrac{1}{2})r\), which is positive unless \(r = 0\).
The result follows.
\end{proof}

\begin{lemma}
\label{helvetica_scenario}
For all \(n \in \N, d | n, r \in \{0,\ldots,n/d-1\}\), we have
\begin{equation}
\max_{\substack{f \in \CF_d (q) \\ f(z) \neq z-1}}
\frac{\left \| \sum_{g \in \CT_{(n)}(q)} \chi^{f,r}(g) \right \|}{\# \CT_{(n)}(q)}
=
O \left ( q^{D(n)-n} \right ).
\end{equation}
\end{lemma}

\begin{proof}
Applying Theorem~\ref{main_lemma_in_intro}, Corollary \ref{MNrule_ncycle}, \eqref{size_of_Fm}, Corollary \ref{rss_cc_and_cycle_type_sizes}, and Lemma \ref{pre_pre_main}, we have
\begin{equation}
\label{final_row}
\frac{
1
}{
\# \CT_{(n)} (q)
}
\left \|
\sum_{g \in \CT_{(n)} (q)} \chi^{f,r} (g)
\right \|
=
\frac{d \cdot \left | \sum_{s | n} \moebius(n/s)(q^s-1) \delta_{q^s - 1 | \ell_{f} [n/d]_{q^d}} \right | }{\sum_{s | n} \moebius(n/s)(q^s-1)}
\end{equation}
for all prime powers \(q\) and \(f \in \CF_d (q)\).
The denominator on the right side of \eqref{final_row} is a degree-\(n\) polynomial in \(q\), independent of \(f\).
However, the numerator on the right side of \eqref{final_row} is not necessarily a polynomial in \(q\) at all, as it also depends on \(\ell_{f}\) which can vary with \(q\), and the sum is inside of an absolute value.
Fortunately, if \(q^n - 1\) does not divide \(\ell_{f} [n/d]_{q^d}\), then
\begin{equation}
\label{D_n_bound}
\begin{aligned}
\left | \sum_{s | n} \moebius(n/s) (q^s-1) \delta_{q^s-1 | \ell_f [n/d]_{q^d}} \right |
& \leq
\sum_{\substack{s | n \\ s < n}}
\left | \moebius (n/s) (q^s-1) \right | \\
& \leq
\sum_{\substack{s | n \\ s < n}}
q^s
<
1 + \sum_{\substack{s | n \\ s < n}}
q^s.
\end{aligned}
\end{equation}
Therefore,
\begin{equation}
\label{final_row_2}
\frac{
1
}{
\# \CT_{(n)} (q)
}
\left \|
\sum_{g \in \CT_{(n)} (q)} \chi^{f,r} (g)
\right \|
\leq
d \cdot \frac{1 + \sum_{s | n, \, s< n} q^s }{\sum_{s | n} \moebius(n/s)(q^s-1)}.
\end{equation}
Observe that \(1 + \sum_{s | n, \, s < n} q^s\) is a degree-\(D(n)\) polynomial in \(q\), and the right side of \eqref{final_row_2} is independent of \(f\).
Moreover, the condition \(q^n-1 \nmid \ell_f [n/d]_{q^d}\) is equivalent to \(f(z) \neq z-1\) because we can assume \(\ell_f \leq q^d-1\), which shows
\[
q^n-1 \mid \ell_f[n/d]_{q^d}
\implies \ell_f = q^d-1
\implies f(1) = 0
\implies f(z) = z-1.
\]
The result now follows from computing the maximum of \eqref{final_row_2} over \mbox{\(f \in \CF_d(q) \setminus \{z-1\}\).}
\end{proof}

\begin{lemma}
\label{char_sum_norm_lemma}
For all \(n \in \N, \, \mu \vdash n, \, d | n\), \(r \in \{0,\ldots,\tfrac{n}{d}-1\}\), we have
\begin{equation}
\max_{f \in \CF_d(q)}
\frac{\left \| \sum_{g \in \CT^\Box_\mu(q)} \chi^{f,r}(g) \right \|}{\gamma_n(q)}
=
O(1).
\end{equation}
\end{lemma}

\begin{proof}
Consider a fixed prime power \(q\) and polynomial \(f \in \CF_d (q)\) to begin.
Apply Theorem~\ref{main_lemma_in_intro} to compute the character values.
Observe that, if some part of \(\mu\) is not divisible by \(d\), then \(\sum_{g \in \CT_\mu^\Box (q)} \chi^{f,r}(g) = 0\), which satisfies the claim.
So assume there exists \(\tilde{\mu} \vdash n/d\) such that \(\mu = d \tilde{\mu}\).
Recall that, by Theorem~\ref{fulmans_characterization}, the conjugacy classes in \(\CT_\mu^\Box (q)\) are in bijection with subsets \(\{h_1,\ldots,h_{\ell(\mu)}\} \subset \CF (q)\) of distinct polynomials with \(\deg h_i = \mu_i\) for each \(i \in \{1,\ldots,\ell(\mu)\}\).
By Theorem~\ref{main_lemma_in_intro}, Corollary~\ref{rss_cc_and_cycle_type_sizes} and the fact that characters are constant on conjugacy classes, we have that
\(
\sum_{g \in \CT_\mu^\Box (q)}
\chi^{f,r} (g)\)
equals
\begin{equation}
\label{everybodys_gotta_learn_sometime}
\frac{\gamma_n(q) (-1)^{\tfrac{n}{d}(d-1)} \chi^{(n/d-r,1^r)}_{\tilde{\mu}}}{\prod_{i=1}^{\ell(\mu)} (q^{\mu_i}-1)}
\sum_{ \substack{\{ h_1,\ldots,h_{\ell(\mu)}\} \subset \CF (q) \\ \deg h_i = \mu_i \forall i} }
\prod_{i=1}^{\ell(\mu)}
\frac{1}{\tilde{\mu}_i}
\sum_{\substack{\alpha_i \in \F_{q^{\mu_i}} \\ h_i(\alpha_i) = 0}}
\te(\alpha_i)^{\ell_f[\tilde{\mu}_i]_{q^d}}.
\end{equation}
We can now separate the sum in \eqref{everybodys_gotta_learn_sometime} according to the degrees of the distinct polynomials \(h_i \in \CF_{\mu_i} (q)\).
Doing so transforms \eqref{everybodys_gotta_learn_sometime} into
\begin{equation}
\frac{\gamma_n(q) (-1)^{\tfrac{n}{d}(d-1)} \chi^{(n/d-r,1^r)}_{\tilde{\mu}}}{\prod_{i=1}^{\ell(\mu)} (q^{\mu_i}-1)}
\prod_{s \geq 1}
\left (
\frac{d}{s}
\right )^{m_s(\mu)}
\sum_{ \{p_1,\ldots,p_{m_s(\mu)}\} \subset \CF_s (q) }
\prod_{i=1}^{m_s(\mu)}
\sum_{\substack{\beta_i \in \F_{q^s} \\ p_i(\beta_i)=0}}
\te(\beta_i)^{\ell_f [s/d]_{q^d}}.
\end{equation}
Computing the norm, applying the triangle inequality, recalling that \(\te\) maps into the unit circle in \(\C\), and applying Corollary~\ref{rss_cc_and_cycle_type_sizes} gives
\begin{align*}
& \left \|
\sum_{g \in \CT^\Box_\mu (q)} \chi^{f,r} (g)
\right \| \\
\leq &
\frac{\gamma_n(q) \left | \chi^{(n/d-r,1^r)}_{\tilde{\mu}} \right |}{\prod_{i=1}^{\ell(\mu)} (q^{\mu_i}-1)}
\prod_{s \geq 1}
\left (
\frac{d}{s}
\right )^{m_s(\mu)}
\sum_{ \{p_1,\ldots,p_{m_s(\mu)}\} \subset \CF_s (q) }
\prod_{i=1}^{m_s(\mu)}
\sum_{\substack{\beta_i \in \F_{q^s} \\ p_i(\beta_i)=0}}
\left \|
\te(\beta_i)^{\ell_f [s/d]_{q^d}}
\right \| \\
= &
\frac{\gamma_n(q) \left | \chi^{(n/d-r,1^r)}_{\tilde{\mu}} \right |}{\prod_{i=1}^{\ell(\mu)} (q^{\mu_i}-1)}
\prod_{s \geq 1}
\left (
\frac{d}{s}
\right )^{m_s(\mu)}
\sum_{ \{p_1,\ldots,p_{m_s(\mu)}\} \subset \CF_s (q) }
s^{m_s(\mu)} \\
= &
\frac{\gamma_n(q) \left |\chi^{(n/d-r,1^r)}_{\tilde{\mu}} \right |}{\prod_{i=1}^{\ell(\mu)} (q^{\mu_i}-1)}
d^{\sum_{s \geq 1} m_s(\mu)}
\prod_{s \geq 1}
\binom{\# \CF_s (q)}{m_s(\mu)} = \# \CT_\mu^\Box (q) \cdot \left | \chi^{(n/d-r,1^r)}_{\tilde{\mu}} \right | \cdot d^{\, \ell(\mu)}.
\end{align*}
Thus,
\begin{equation}
\label{constant_bound}
\frac{1}{\gamma_n(q)}
\left \|
\sum_{g \in \CT_\mu^\Box (q)} \chi^{f,r} (g)
\right \|
\leq
\frac{\# \CT_\mu^\Box (q)}{
\gamma_n (q)}
\cdot \left | \chi^{(n/d-r,1^r)}_{\tilde{\mu}} \right | \cdot d^{\, \ell(\mu)}.
\end{equation}
The right side of \eqref{constant_bound} does not depend on \(f\), which implies
\begin{equation}
\label{constant_bound_2}
\max_{f \in \CF_d (q)}
\frac{1}{\gamma_n(q)}
\left \|
\sum_{g \in \CT_\mu^\Box (q)} \chi^{f,r} (g)
\right \|
\leq
\frac{\# \CT_\mu^\Box (q)}{
\gamma_n (q)}
\cdot \left | \chi^{(n/d-r,1^r)}_{\tilde{\mu}} \right | \cdot d^{\, \ell(\mu)}.
\end{equation}
Moreover, by Corollary~\ref{dense}, for sufficiently large \(q\), the right side of \eqref{constant_bound_2} is arbitrarily close to the constant value \(|\chi_{\tilde{\mu}}^{(n/d-r,1^r)}| \cdot d^{\ell(\mu)} / z_\mu\).
The result follows.
\end{proof}

\subsection{Proof of probabilistic result}

We can now prove our probabilistic result, Theorem~\ref{intro_version_p}.
Recall that it states
\begin{equation}
\lim_{q \to \infty}
p_{k,\mu} (q)
=
\lim_{q \to \infty}
p_{k,\mu}^\Box (q)
=
\frac{1}{z_\mu}
\end{equation}
for all \(n \in \N\) and \(\mu \vdash n\).

\begin{proof}[Proof~of~Theorem~\ref{intro_version_p}]
We will first prove that \(\lim_{q \to \infty} p_{k,\mu}^\Box (q) = 1 / z_\mu\).
From this, it follows that \(\lim_{q \to \infty} p_{k,\mu} (q) = 1 / z_\mu\) because, for all prime powers \(q\), we have \(p_{k,\nu} (q) \geq p_{k,\nu}^\Box (q)\) for all \(\nu \vdash n\) and \(\sum_{\nu \vdash n} p_{k,\nu} (q) = 1 = \sum_{\nu \vdash n} 1/z_\nu\).

Consider the following formulation of \(p_{k,\mu}^\Box (q)\).
By its definition \eqref{pBox_defs} and by Corollary~\ref{simpler_ff_in_situ_corollary}, we have
\begin{equation}
\label{pBox_formulation}
p_{k,\mu}^\Box (q) =
\sum_{d | n}
\sum_{r = 0}^{\tfrac{n}{d} - 1}
\sum_{f \in \CF_d (q)}
\left (
\frac{\sum_{g \in \CT_{(n)}(q)} \chi^{f,r}(g)}{\# \CT_{(n)}(q)}
\right )^k
\left (
\frac{\sum_{h \in \CT^\Box_\mu(q)} \chi^{f,r}(h)}{\gamma_n(q) \cdot \deg_{n,d,r}(q)^{k-1}}
\right ).
\end{equation}
We want to compute \(\lim_{q \to \infty} p_{k,\mu}^\Box (q)\), but the index set for the summation over \(f \in \CF_d (q)\) in \eqref{pBox_formulation} itself depends on \(q\).
Therefore, for each \(d | n\), \(r \in \{0,\ldots,n/d-1\}\), and \(f \in \CF_d (q)\) we define
\begin{equation}
\Phi_{k,\mu,d,r} (q)
=
\sum_{f \in \CF_d (q)}
\left (
\frac{\sum_{g \in \CT_{(n)}(q)} \chi^{f,r}(g)}{\# \CT_{(n)}(q)}
\right )^k
\left (
\frac{\sum_{h \in \CT^\Box_\mu(q)} \chi^{f,r}(h)}{\gamma_n(q) \cdot \deg_{n,d,r}(q)^{k-1}}
\right )
\end{equation}
so that
\begin{equation}
p^\Box_{k,\mu} (q)
=
\sum_{d | n}
\sum_{r = 0}^{\tfrac{n}{d}-1}
\Phi_{k,\mu,d,r} (q),
\end{equation}
where the number of terms in the summation is fixed, even as \(q\) varies.
Theorem~\ref{intro_version_p} now follows from Lemma~\ref{Phi_behavior} below, which computes the limiting behavior of \(\Phi_{k,\mu,d,r} (q)\) for each \(d | n\) and \(r \in \{0,\ldots,n/d-1\}\).
\end{proof}

\begin{lemma}
\label{Phi_behavior}
For all \(n,k \in \N, \mu \vdash n, d | n\), and \(r \in \{0,\ldots,n/d-1\}\), we have
\begin{equation}
\lim_{q \to \infty}
\Phi_{k,\mu,d,r} (q)
=
\begin{cases}
0         & \text{ if } d > 1 \text{ or }  r > 0, \\
1 / z_\mu & \text{ if } d = 1 \text{ and } r = 0.
\end{cases}
\end{equation}
\end{lemma}

\begin{proof}
Consider first the case that \(d > 1\) and \(r\) is arbitrary.
Observe that
\(
\|
\Phi_{k,\mu,d,r} (q)
\|
\)
is bounded above by
\begin{equation}
\label{naive_max_bound}
\frac{
\# \CF_d (q)
}{
\deg_{n,d,r} (q)^{k-1}
}
\cdot
\max_{f \in \CF_d (q)}
\left (
\frac{\left \| \sum_{g \in \CT_{(n)}(q)} \chi^{f,r}(g) \right \|}{\# \CT_{(n)}(q)}
\right )^k
\cdot
\max_{f \in \CF_d(q)}
\frac{\left \| \sum_{h \in \CT^\Box_\mu(q)} \chi^{f,r}(h) \right \|}{\gamma_n(q)}
\end{equation}
for all prime powers \(q\).
We proceed to investigate the asymptotic dependence on \(q\) of \eqref{naive_max_bound}.
Recall from \eqref{size_of_Fm} that \(\# \CF_d (q) = O(q^d)\).
Combining this with Lemmas \ref{degree_limit_lemma},
\ref{helvetica_scenario},
and
\ref{char_sum_norm_lemma},
we have
\begin{equation}
\| \Phi_{k,\mu,d,r} (q) \|
=
O \left ( q^{d + k (D(n) - n) - (k-1) \cdot e(n,d,r)} \right ).
\end{equation}
By hypothesis, \(k \geq 2\), implying \(d  + k \cdot (D(n) - n) \leq d - k \cdot n/2 \leq 0\).
Moreover, by Lemma~\ref{degree_limit_lemma}, \((k-1)\cdot e(n,d,r) > 0\).
It follows that \(\lim_{q \to \infty} \Phi_{k,\mu,d,r} (q) = 0\) if \(d > 1\).

Next, consider the case \(d = 1\) and \(r > 0\).
Observing that \(z-1 \in \CF_1 (q)\) for all prime powers \(q\) and applying Theorem~\ref{sn_stuff}, we can rewrite \(\Phi_{k,\mu,1,r} (q)\) as
\begin{align}
\Phi_{k,\mu,1,r} (q)
& =
\frac{\# \CT_\mu^\Box}{\gamma_n (q)}
\cdot
\frac{(-1)^{rk} \chi^{(n-r,1^r)}_\mu}{\left ( q^{\binom{r+1}{2}} {\genfrac[]{0pt}{1}{n - 1}{r}}_{q} \right )^{k-1}}
\label{the_z_minus_1_term}
\\
& + \label{not_z_minus_1_term}
\sum_{\substack{f \in \CF_1 (q) \\ f(z) \neq z-1}}
\left (
\frac{\sum_{g \in \CT_{(n)}(q)} \chi^{f,r}(g)}{\# \CT_{(n)}(q)}
\right )^k
\left (
\frac{\sum_{h \in \CT^\Box_\mu(q)} \chi^{f,r}(h)}{\gamma_n(q) \cdot \deg_{n,d,r}(q)^{k-1}}
\right ).
\end{align}
We repeat the same analysis as before, but apply it only to \eqref{not_z_minus_1_term}.
Observe that \eqref{not_z_minus_1_term} is bounded above by
\begin{equation}
\label{not_z_minus_1_term_bound}
\frac{
(\# \CF_1(q)) - 1
}{
\deg_{n,d,r} (q)^{k-1}
}
\cdot
\max_{\substack{f \in \CF_1 (q) \\ f(z) \neq z-1}}
\left (
\frac{\left \| \sum_{g \in \CT_{(n)}(q)} \chi^{f,r}(g) \right \|}{\# \CT_{(n)}(q)}
\right )^k
\cdot
\max_{\substack{f \in \CF_1 (q) \\ f(z) \neq z-1}}
\frac{\left \| \sum_{h \in \CT^\Box_\mu(q)} \chi^{f,r}(h) \right \|}{\gamma_n(q)}
\end{equation}
Applying Lemmas \ref{degree_limit_lemma}, \ref{helvetica_scenario}, and \ref{char_sum_norm_lemma} again, we see that \eqref{not_z_minus_1_term_bound} is
\begin{equation}
\label{d_1_non_z_minus_1}
O \left (
q^{1 + k (D(n) - n) - (k-1) \cdot e(n,1,r)}
\right ).
\end{equation}
As before, \(1 + k \cdot (D(n) - n) - (k-1) \cdot e(n,1,r)\) is negative.
It follows that the limit as \(q \to \infty\) of \eqref{not_z_minus_1_term} is zero if \(d = 1\) and \(r > 0\).
Applying Corollary~\ref{dense} to \eqref{the_z_minus_1_term}, we can conclude that \(\lim_{q \to \infty} \Phi_{k,\mu,1,r} (q) = 0\) if \(d = 1\) and \(r > 0\).

Finally, we consider the case \(d = 1\) and \(r = 0\).
Carrying out the same analysis as in the previous paragraph, we see that
\[
\| \Phi_{k,\mu,1,0} (q) \|
=
\frac{\# \CT_\mu^\Box (q)}
{\gamma_n (q)}
+
O \left (
q^{1 + k(D(n) - n)}
\right ).
\]
Observe that \(1 + k(D(n) - n) < 0\) even if \(n = 1\) due to the fact that \(k \geq 2\) and \(D(1)=0\).
The result now follows from Corollary~\ref{dense}.
\end{proof}

\section{Further work}
\label{outro}

\subsection{Polynomiality results}

We discuss some results regarding how similar \(g_{k,\mu}(q)\) is to a polynomial for various choices of \(\mu\).
Recall that \(\gamma_n (q), P_{n,k} (q), \deg_{n,d,r} (q)\), and \(D_{n,k,d} (q)\) are all rational in \(q\) with rational coefficients.

\begin{corollary}[to~Thm.~\ref{re_main}]
\label{full_polynomiality_for_some}
Suppose \(n, k, \ell \in \N\) with \(n > 2\) and \(\ell > 1\).
If \(\mu = (\mu_1,\ldots,\mu_{\ell})\) with \(\mu_{\ell-1} > \mu_\ell = 1\), then \(g_{k,\mu}^\Box (q)\) is a polynomial in \(q\) with rational coefficients.
\end{corollary}

\begin{proof}
Theorem~\ref{re_main} implies that \(n^k z_\mu \cdot g^\Box_{k,\mu} (q)\) is a rational function of \(q\) with integer coefficients which takes on integral values infinitely many times.
Thus \(n^k z_\mu \cdot g^\Box_{k,\mu} (q)\) is an integer polynomial in \(q\).
Dividing by \(n^k z_\mu \in \N\) gives the result.
\end{proof}

Next, we prove Corollary~\ref{polynomiality_for_nu_n}, which states that \(g_{k,(n)} (q)\) is a quasipolynomial in \(q\) of quasiperiod \(n\) in the case that \(n\) is prime.
Note that \(g_{1,(n)} (q) = \# \CT_{(n)} (q)\), which is a polynomial in \(q\), independent of \(n\) or the congruence class of \(q\).
However, when \(k \geq 2\), Corollary~\ref{polynomiality_for_nu_n} has more to say.

\begin{proof}[Proof of Corollary~\ref{polynomiality_for_nu_n}]
We will apply a similar reasoning to that stated in the proof of Corollary~\ref{full_polynomiality_for_some}.
Recall that the function \(C_{n,k,c}(q)\) is not rational in \(q\) in general,
which prevents \(g_{k,(n)} (q)\) from being rational.
Define
\begin{equation}
M_i = \{ q \, \text{ prime power} : q \equiv i \pmod{n}\}, \quad \text{for } i \in \{0,1,\ldots,n-1\}.
\end{equation}
The result will follow once we can show that, for each \(c | n\) and \(i \in \{0,\ldots,n-1\}\), we have that \(C_{n,k,c}(q)\) becomes a polynomial in \(q\) when restricted to \(M_i\).

In order to do this, it suffices to show that, for each \(i \in \{0,\ldots,n-1\}\) and choice of \(c, s_1,\ldots,s_k | n\),
\begin{equation}
\label{objective_lcm}
\lcm
\left (
\frac{q^n-1}{q^c-1},
q^{s_1} - 1,
\ldots,
q^{s_k} - 1
\right )
\end{equation}
agrees with some polynomial on \(M_i\).
Since \(n\) is prime, we have \(c,s_1,\ldots,s_k \in \{1,n\}\).
Furthermore, if any \(s_i = n\), we have that \eqref{objective_lcm} equals \(q^n-1\), a fixed polynomial in \(q\),
independent of \(c\) or the congruence class of \(q\).
Therefore, for each choice of \(c \in \{1,n\}\), we need only consider the case in which \(s_1 = \cdots = s_k = 1\) and hence must show that
\begin{equation}
\label{objective_lcm_2}
\lcm
\left (
\frac{q^n-1}{q^c-1},
q - 1
\right )
\end{equation}
is a polynomial on each \(M_i\).

Observe that, if \(c = n\), then \eqref{objective_lcm_2} equals \(q-1\), a fixed polynomial in \(q\), independent of the congruence class of \(q\).
Therefore, we now need only consider the case \(c = s_1 = \cdots = s_k = 1\) and hence must show that
\begin{equation}
\label{objective_lcm_3}
\lcm
\left (
\frac{q^n-1}{q-1},
q - 1
\right )
\end{equation}
is a polynomial on each \(M_i\).

Let \(i \in \{0,\ldots,n-1\}\)
and assume \(q = na + i\) for some \(a \in \N\).
We can compute \eqref{objective_lcm_3} as
\begin{align}
\lcm
\left (
\frac{q^n-1}{q-1},
q-1
\right )
& =
\frac{q^n-1}{
\gcd
\left (
[n]_q, q-1
\right )
}
=
\frac{q^n-1}{
\gcd \left (
n,q-1
\right )
} \nonumber \\
& =
\frac{q^n-1}{\gcd(n,na+i-1)}
= \frac{q^n-1}{\gcd(n,i-1)}
=
\begin{cases}
\frac{q^n-1}{n} & i=1 \\
q^n-1           & i \neq 1.
\end{cases}
\label{grand_finale_i_guess}
\end{align}
The result now follows from the fact that \eqref{grand_finale_i_guess} is a fixed polynomial in \(q\) for each fixed \(i \in \{0,\ldots,n-1\}\).
\end{proof}

\begin{example}
We now use the main results of the paper to write down alternate formulas for \(g_{2,(2)} (q)\) and \(g_{2,(3)} (q)\).
Note that Theorem~\ref{nu_n_part} provides an explicit formula while Theorem~\ref{intro_version_p} determines the degree of the polynomials \(f_0,\ldots,f_{n-1}\) mentioned in Corollary~\ref{polynomiality_for_nu_n}.
First, for \(n = 2\), we have
\begin{equation}
g_{2,(2)}(q)
=
\frac{q (q - 1)^3 (q^4 - 3q^3 + 4q^2 - \tfrac{1}{2}q - \tfrac{1}{2})}{8}
+(-1)^q
\cdot
\frac{q(q + 1)(q - 1)^3}{16}
\end{equation}
for all prime powers \(q\).
Furthermore, for \(n = 3\), define polynomials
\begin{align*}
f_0 (q) & = \frac{q^6 (q + 1)^2 (q - 1)^4  (q^6 - 4q^4 + 3q^3 + 5q^2 - 9q + 1)}{27}, \\
f_1 (q) & = \frac{q^3 (q + 1)  (q - 1)^5 (q^9 + 2q^8 - 2q^7 - 3q^6 + 5q^5 + q^4 - 9q^3 - 4q^2 - 2q + 2)}{27}, \\
f_2 (q) & = \frac{q^6 (q + 1)^2 (q - 1)^4  (q^6 - 4q^4 + 3q^3 + 5q^2 - 9q + 1)}{27}.
\end{align*}
Letting \(\zeta = e^{2 \pi i / 3}\), define
\begin{equation*}
P_1 = \frac{f_0 + \zeta^2 f_1 + \zeta f_2}{3},
\quad
P_2 = \frac{f_0 + \zeta f_1 + \zeta^2 f_2}{3},
\quad
P_3 = \frac{f_0 + f_1 + f_2}{3}.
\end{equation*}
Finally, we have
\begin{equation}
g_{2,(3)} (q) = \zeta^q P_1(q) + \zeta^{2q} P_2(q) + P_3(q)
\end{equation}
for all prime powers \(q\).
\end{example}

\begin{remark}
\label{bummer}
Data suggest that, in general, \(g_{k,(n)}(q)\) is not a quasipolynomial when \(n\) is not prime.
In fact, because Theorem~\ref{intro_version_p} controls the degree of any polynomial that might agree with an infinite family of values of \(g_{k,(n)} (q)\) for fixed \(n,k\), one can prove in specific instances that such a polynomial does not exist.
For instance, \(g_{2,(4)} (q)\) is not a polynomial on the prime powers congruent to \(2 \pmod{4}\).
\end{remark}

\subsection{Open problems}

In this section, we list some open problems.
Of course, one can continue our present line of research by looking for explicit formulas for \(g_{k,\mu} (q)\) and \(g_{k,\mu}^\Box (q)\) for cases not yet settled by this paper.
However, we also present the following problems associated with strengthening the existing results.

We start with the observation that Theorems~\ref{re_main} and \ref{nu_n_part} do not answer the question of how products of regular elliptic elements are distributed among the individual conjugacy classes that comprise the various cycle types.
In particular, given a fixed regular elliptic element \(c \in \CT_{(n)}\), computing \(g_{k,(n)} (q)\) does not necessarily help one count the factorizations \(c = t_1 \cdots t_k\) with \(t_1,\ldots,t_k \in \CT_{(n)}\).
Therefore, we propose the following problem.

\begin{problem}
Refine Theorems \ref{re_main} and \ref{nu_n_part} to the level of conjugacy classes.
\end{problem}

Next, we recall Corollary~\ref{polynomiality_for_nu_n}, which says if \(n\) is prime, then \(g_{k,(n)} (q)\) is a quasipolynomial.
Data suggest that, if \(n\) is not prime, \(g_{k,(n)} (q)\) still agrees with a polynomial at least on the set of prime powers congruent to \(1\) modulo \(n\).
This is a considerably weaker result, but it suggests \(g_{k,(n)} (q)\) might have some nice structure, even though Remark~\ref{bummer} points out that we cannot expect quasipolynomiality in general.

\begin{problem}
Prove that, even when \(n\) is not prime, there exists a polynomial \(f_1 \in \Q[x]\), such that \(g_{k,(n)} (q) = f_1(q)\) for all prime powers \(q \equiv 1 \pmod{n}\).
\end{problem}

We conclude with a problem about \(q\)-analogues.
As noted by \eqref{crazy_q_analogue_idea}, for some choices of \(\mu \vdash n\) and after appropriately normalizing, \(g_{k,\mu}^\Box (q)\) appears to be a \(q\)-analogue of \(g_{k,\mu}\) in the traditional \(q \to 1\) sense.
Unfortunately, it is not clear whether \(g_{k,(n)} (q)\) exhibits the same behavior.

\begin{problem}
Establish a precise way in which \(g_{k,(n)} (q)\) is a \(q\)-analogue of \(g_{k,(n)}\).
\end{problem}

\section*{Acknowledgments}

The author thanks
Sara Billey,
Jia Huang,
Joseph Kung,
Joel Lewis,
Alejandro Morales, and
Vic Reiner,
for their help in writing this paper.

\bibliographystyle{plain}
\bibliography{./the}

\end{document}